\documentclass[12pt]{amsart}
\usepackage{amsmath,amssymb,latexsym,esint,cite,mathrsfs}
\usepackage{verbatim,wasysym}
\usepackage[left=2.6cm,right=2.6cm,top=2.8cm,bottom=2.8cm,
marginparwidth=20mm, marginparsep=3mm]{geometry}
\usepackage{tikz,enumitem,graphicx, subfig, microtype, color}
\usepackage{epic,eepic}

\allowdisplaybreaks
\usepackage[colorlinks=true,urlcolor=blue, citecolor=red,linkcolor=blue,
linktocpage,pdfpagelabels, bookmarksnumbered,bookmarksopen]{hyperref}
\usepackage[english]{babel}

\numberwithin{equation}{section}

\newtheorem{thm}{Theorem}[section]

\newtheorem{cor}[thm]{Corollary}
\newtheorem{pro}[thm]{Proposition}

\newtheorem{exa}[thm]{Example}

\theoremstyle{definition}
\newtheorem{rem}[thm]{Remark}

\def\FORALL{\text{ for all }}
\def\FOR{\text{ for }}
\def\IN{\text{ in }}
\def\AND{\text{ and }}
\def\ON{\text{ on }}
\def\IF{\text{ if }}
\def\ol{\overline}
\def\tim{\times}
\def\bald{\begin{aligned}}
\def\eald{\end{aligned}}
\def\ot{\otimes}
\def\bmat{\begin{pmatrix}}\def\emat{\end{pmatrix}}
\def\hr{\hat\rho}

\newcommand{\dist}{{\rm dist}}

\newcommand{\N}{\mathbb{N}}

\newcommand{\R}{\mathbb{R}}

\def\ga {\gamma}

\def\cS{\mathcal{S}}

\def\bproof{\begin{proof}}
\def\eproof{\end{proof}}
\def\gd{\delta} \def\gth{\theta}\def\x{\hat x}
\def\lan{\langle} \def\ran{\rangle}
\def\stm{\setminus}
\def\ep{\varepsilon}
\def\ON{\text{ on }}
\def\beq{\begin{equation}}
\def\eeq{\end{equation}}
\def\x{\hat x} \def\y{\hat y} 
\def\ga{\alpha} \def\gl{\lambda} \def\gL{\Lambda}
\def\fr{\frac}\def\beq{\begin{equation}}
\def\eeq{\end{equation}} 
\def\du#1{\left\langle #1\right\rangle}

\DeclareMathOperator\USC{USC}
\DeclareMathOperator\LSC{LSC}
\DeclareMathOperator\Lip{Lip}

\title[]{Balance between degenerate elliptic operators and coercive Hamiltonians}

\author[I. Birindelli]{Isabeau Birindelli}
\address[I. Birindelli]{Dipartimento di Matematica Guido Castelnuovo, Sapienza 
Universit\`a di Roma, Piazzale Aldo Moro 5, Roma, Italy.}
\email{isabeau@mat.uniroma1.it}

\author[G. Galise]{Giulio Galise}
\address[G. Galise]{Dipartimento di Matematica Guido Castelnuovo, Sapienza 
Universit\`a di Roma, Piazzale Aldo Moro 5, Roma, Italy.}
\email{galise@mat.uniroma1.it}

\author[H. Ishii]{Hitoshi Ishii}
\address[H. Ishii]{Institute for Mathematics and Computer Science, Tsuda University, 2-1-1 Tsuda,
Kodaira, Tokyo 187-8577, Japan.}
\email{hitoshi.ishii@waseda.jp}

\def\aka#1{\textcolor{red}{#1}}

\begin{document}

\begin{abstract}
For $p>1$, we consider the boundary value problem
for fully nonlinear degenerate elliptic equations $-\lambda_i(D^2u)+|D u|^p+\gamma u=f(x)$ in bounded domains with Dirichlet or boundary blow-up conditions; here $\lambda_i(D^2u)$ denotes the $i$-th eigenvalue of the Hessian. We study existence and nonexistence of solutions together with the asymptotic behaviour of the solutions when $\gamma$ goes to zero.  A priori Lipschitz estimates play an important role. The interplay between the operator’s degeneracy and the superlinear growth of the Hamiltonian gives rise to phenomena that are very different depending on which of the two terms dominates, e.g. the ergodic dichotomy takes place only when $i=N$, while new phenomena arise for $i<N$ in which case, under  mild conditions, solutions that blow up even in just one point do not exist,  and conditions on the size of $f$  must be imposed for the existence of solutions to the  Dirichlet problem with homogeneous boundary condition. 
\end{abstract}

\maketitle 

\tableofcontents
\allowdisplaybreaks

\section{Introduction}

In this paper, we investigate the existence of viscosity sub- and supersolutions for degenerate elliptic equations featuring a superlinear Hamiltonian of the form:
\begin{equation}\label{into1}
- F(D^2u)+|D u|^p=f(x) \quad \text{in } \Omega,
\end{equation}
where $\Omega$ is a domain in $\mathbb{R}^N$ and $p>1$. Specifically, we focus on the fully nonlinear degenerate operators defined by
$$F(X):=\lambda_i(X),$$
where $\lambda_1(X) \leq \dots \leq \lambda_N(X)$ denote the eigenvalues of the symmetric matrix $X \in \mathcal{S}(N)$ arranged in nondecreasing order. While our primary focus is on these single-eigenvalue operators, several of our results extend to linear combinations of eigenvalues, such as e.g. the truncated Laplacians:
\begin{equation}\label{pk}
F(X):=\mathcal{P}_k^-(X)=\sum_{i=1}^k\lambda_i(X), \quad\text{or}\quad F(X):=\mathcal{P}_k^+(X)=\sum_{i=N-k+1}^N\lambda_i(X).
\end{equation}
We will explicitly remark on these extensions throughout the text whenever they apply. Additionally, although some of our theorems hold for a wider range of exponents, our analysis is mainly on the case $p\in (1,2]$.

\medskip
For uniformly elliptic operators as the standard Laplacian $F(D^2u)=\Delta u$, the question of solvability has been thoroughly studied, and the existence of solutions is well known to be intimately connected to the \lq\lq size\rq\rq\ of the source term $f$ (see, e.g., \cite{HMV}). This problem is also related to the associated ergodic framework, as established in the seminal works \cite{LL} and \cite{Por}, which proved the following classical dichotomy for $p\in (1,2]$:

\noindent on the one hand, the existence of a solution $u_0$ to the Dirichlet problem associated to \eqref{into1} with homogeneous boundary conditions $u=0$ on $\partial\Omega$ implies that, as $\gamma\rightarrow 0^+$, the unique solutions $u_\gamma$ of the  problems
\begin{equation}\label{int2}
\left\{\begin{array}{cl}
-F(D^2u)+|D u|^p+\gamma u=f(x) & \text{in $\Omega$}\\
u=0 & \text{on $\partial\Omega$}
\end{array}\right.
\end{equation}
converge uniformly to $u_0$; on the other hand, if such a solution of \eqref{into1} doesn't exist, then the family $\{u_\gamma\}$ becomes unbounded; but, along a suitable sequence $\gamma_n \to 0^+$, the functions $v_{\gamma_n}=u_{\gamma_n}+\left\|u_{\gamma_n}^-\right\|_\infty$ converge locally uniformly in $\Omega$ to a function $v$ satisfying the ergodic problem:
\begin{equation}\label{interg}
\left\{\begin{array}{cl}
-F(D^2v)+{|Dv|}^p=f(x)+C & \text{in $\Omega$}\\
v=+\infty & \text{on $\partial \Omega$},
\end{array}\right.
\end{equation}
for a unique constant $C$ known as the ergodic constant. Generalizations of such a result to the nonlinear setting have been obtained for the $p$-Laplacian in \cite{LP} and for some singular/degenerate fully nonlinear operators in \cite{BDL1}. 

\medskip
In the degenerate settings investigated here, the interplay between the operator's degeneracy and the superlinear growth of the Hamiltonian gives rise to different phenomena. Depending on whether the superlinear gradient term dominates or not the second-order degenerate term,  new behaviours emerge, or alternatively, a counterpart to the classical elliptic dichotomy is recovered. 

In particular for the operator $F(D^2u)=\lambda_i(D^2u)$, the case $i<N$ is different from the standard theory, whereas the top eigenvalue $i=N$ retains many analogies with the uniformly elliptic framework. 

To highlight this contrast, we begin by stating an unusual nonexistence result for large solutions (i.e., blowing up at the boundary):

\begin{thm}\label{noblowupint}
Let $p>1$, $i<N$, and let $C$ be any constant. Suppose $\Omega$ is an open subset of $\mathbb{R}^N$ satisfying the interior sphere condition at some point $x_0\in\partial\Omega$. Then, there exist no viscosity subsolutions $u\in\USC(\Omega)$ of the equation
\begin{equation}\label{blowuppb}
-\lambda_i(D^2u)+|D u|^p=C \quad\,\text{in $\Omega$},
\end{equation}
satisfying the boundary blow-up condition
\begin{equation}\label{binfty}
\lim_{x\to x_0}u(x)=+\infty.
\end{equation}
\end{thm}

Theorem \ref{noblowupint} implies that for any 
domain $\Omega$, there are no solutions that blow-up on the entire boundary; structurally ruling out the existence of an ergodic constant when $i<N$. This marks a difference from the case $i=N$ where, much like the uniformly elliptic case, a unique constant $C$ exists for which boundary blow-up solutions can be constructed. Moreover, when $\Omega$ is a ball and $f$ is constant, this constant can be determined explicitly (see Section \ref{lambdaN}).

Consequently it is important to understand, always for $i<N$, the relationship between the Dirichlet problems:
\begin{equation}\tag{DP$_0$}\label{DPI0}
\left\{\begin{array}{cl}
-\lambda_i(D^2u)+|D u|^p=f(x) & \text{in $\Omega$}\\
u=0 & \text{on $\partial \Omega$}
\end{array}\right.
\end{equation}
and 
\begin{equation}\tag{DP$_\gamma$}\label{DPI}
\left\{\begin{array}{cl}
-\lambda_i(D^2u)+|D u|^p+\gamma u=f(x) & \text{in $\Omega$}\\
u=0 & \text{on $\partial\Omega$}
\end{array}\right.
\end{equation}
with $\gamma>0$. From the viewpoint of solvability, we show that problems \eqref{DPI0} and \eqref{DPI} are equivalent in the model case where $\Omega = B_R$ and $f \equiv -C$ for a positive constant $C$: 

\begin{pro}\label{DPvsDP0}
Let $\Omega=B_R$, $C>0$, $p>1$, and set $f(x)=-C$. Then, the boundary value problem \eqref{DPI0} admits a solution if, and only if, \eqref{DPI} admits a solution. Moreover, the solutions to both \eqref{DPI0} and \eqref{DPI} are unique, and they exist if, and only if, the radius satisfies $R\leq\bar R:=\frac{(p-1)^\frac{p-1}{p}}{pC^\frac{p-1}{p}}$.
\end{pro}

This solvability threshold again highlights the contrast with the uniformly elliptic framework, while remaining consistent with the nonexistence of blow-up solutions established in Theorem \ref{noblowupint}.

When the domain is a ball and the solutions are assumed to be radial, the necessary and sufficient conditions for the existence of solutions can be made explicit. For instance, the existence of a radial subsolution implies that a radial source $f(r)$ must satisfy the bound:
$$-\frac{p-1}{(pr)^{\frac{p}{p-1}}}\leq f(r) \quad \text{for } r\in(0,R).$$
Conversely, if the above inequality holds strictly, the condition becomes sufficient to guarantee existence (see Proposition \ref{4marprop2} for the precise statements). 
Since radial $C^2$ functions $u(x)=U(|x|)$ have Hessians whose eigenvalues are $\frac{U'}{r}$ and $U''$, the operator reduces to $F(D^2u)=G(\frac{U'}{r},U'')$ for an appropriate function $G$. This allows us to construct explicit $C^2$ radial solutions via ODE techniques. Though not strictly required for the core results of this paper, we establish,  in Proposition \ref{eqv-rad-gen}, an interesting  equivalence between the PDE and ODE formulations for viscosity radial sub- and supersolutions.

\medskip

As is standard in the analysis of nonlinear PDEs, the key tools for establishing these results are a priori estimates. Our first two estimates depend on the $L^\infty$ norm of the solution:
\begin{itemize}
\item For $p>1$ and any $i\in\{1,\dots,N\}$, we establish the optimal Lipschitz regularity of radial subsolutions away from the origin.
\item For $p\in (0,2]$ and $i=N$, we prove local interior Lipschitz regularity for supersolutions. Interestingly, this property is highly specific to the operator $\lambda_N(D^2u)$ and fails for general uniformly elliptic operators; we provide counterexamples to demonstrate this sharpness.
\end{itemize}

On the other hand, to handle the asymptotic limit in the ergodic problem where the solutions $u_\gamma$ explode while $\gamma u_\gamma$ remains bounded, we establish local Lipschitz estimates that depending on  $\gamma \|u_\gamma\|_\infty$ and not on $\|u_\gamma\|_\infty$. 
This is carried out in Proposition \ref{prop1regularity}. Similar a priori bounds that are helpful  when the  $L^\infty$ norm  is unbounded but $\gamma \|u_\gamma\|_\infty$ is bounded, have been obtained in the linear degenerate elliptic setting in \cite{CDLP,B2,AT}, 
and for fully nonlinear equations in \cite{BDL1}.

Before closing this introduction,  we mention that there has been a line of research stemming from the study of homogenization for Hamilton-Jacobi equations (see \cite{LPV}) that adopts a similar yet slightly different perspective. A major direction within this field focuses on the ergodic approximation (ergodic convergence or the \lq\lq vanishing discount\rq\rq\ limit) of Hamilton-Jacobi equations; under assumptions regarding the convexity of the equation—characteristic of Bellman-type equations—researchers have established that the solution to equation  \eqref{int2}, which involves a discount factor $\gamma > 0$, converges to the solution of equation \eqref{interg} as $\gamma \to 0^+$. Note that the boundary conditions for equations \eqref{int2} and \eqref{interg} are chosen appropriately,  such as periodic, state-constraint, Neumann, and Dirichlet conditions . For further details on the development of research in this direction—ranging from 
Hamilton-Jacobi equations to second-order degenerate elliptic Bellman equations—the reader is invited to refer to \cite{DFIZ, MT, IMT1, IMT2}.

\subsection*{Structure of the paper} 
The paper is organized as follows: Section \ref{LipReg} is dedicated to local Lipschitz regularity and a priori estimates. Section \ref{radial} investigates necessary and sufficient conditions for the existence of radial $C^2$ solutions. Section \ref{theproblem} analyzes the case $F(D^2u)=\lambda_i(D^2u)$ for $i<N$, where the most unconventional phenomena occur. Section \ref{ergodic} focuses on the top eigenvalue $i=N$, proving the ergodic dichotomy and related existence results. Finally, the Appendix provides the proof of the equivalence between the PDE and ODE formulations for radial viscosity solutions.

\subsection*{Notations and preliminaries}
Henceforth $\cS(N)$ will denote the space of $N\times N$ real symmetric matrices.
For $X\in\cS(N)$, we denote by 
$$
\lambda_1(X)\leq\ldots\leq\lambda_N(X)
$$
the eigenvalues of $X$ arranged in nondecreasing order. The norm of $X$ is 
$$
|X|=\sup_{|v|=1}\left| Xv\cdot v\right|,
$$
where $w\cdot v$ (also, interchangeably, $\lan w,v\ran$) denotes the standard inner product of $w,v\in\R^N$.
Observe that 
$$
|\lambda_i(X)-\lambda_i(Y)|\leq 
|X-Y|\quad\text{for $X,Y\in\cS(N)$.}
$$
Indeed, if we write 
$$
V_i=\{v_1,\ldots, v_i|\, v_k\in\R^N, \ v_k\cdot v_l=\delta_{kl}  \quad\text{for $k,l\in\{1,\ldots,i\}$}\},
$$
then
\[\begin{aligned}
\lambda_i(X)&=\inf_{V_i}\sup_{v\in V_i} Xv\cdot v,
\\
|\lambda_i(X)-\lambda_i(Y)|&\leq\sup_{V_i}\sup_{v\in V_i} |Xv\cdot v-Yv\cdot v|
=\sup_{V_i}\sup_{v\in V_i} |(X-Y)v\cdot v|
\\&
\leq |X-Y|. 
\end{aligned}\]

For any positive constant $a$, the nonlinear operators $F(X)=a\lambda_i(X)$ 
is positively homogeneous of degree $1$,  
degenerate elliptic and Lipschitz continuous: 
\begin{align}
&\label{F0}
F(tX)=tF(X)\quad \FOR t\geq 0, X\in\cS(N),
\\&X\leq Y \ \implies\ F(X) \leq F(Y),
\label{F1}
\\& 
|F(X)-F(Y)|\leq a|X-Y|.  \label{F2}
\end{align}
%
\section{Interior Lipschitz regularity results}\label{LipReg}

\subsection{A result for radial subsolutions}

\begin{pro}\label{prop2regularity}
Let $u\in\USC(B_R\backslash\left\{0\right\})$  be a radial and locally bounded viscosity subsolution of 
$$
-\lambda_i(D^2u)+|Du|^p=C\quad\,\text{in\, $B_R\backslash\left\{0\right\}$},
$$
where $i\in\left\{1,\ldots,N\right\}$, $p>1$ and $C\geq0$. Then for any  $\omega\subset\subset\omega'\subset\subset B_R\backslash\left\{0\right\}$, there exists a positive constant  $L=L(p,C,\dist(\omega,\partial\omega'),\left\|u\right\|_{L^\infty(\omega')})$
such that 
\[
|u(x)-u(y)|\leq L|x-y| \quad\text{for $x,y\in \omega$.}
\]
\end{pro}
\begin{proof}
Since $u(x)=U(r)$ by assumption, it is sufficient to prove that for any fixed $0< R_1< R_2<R$, there exist positive $\delta$ and $L$ such that for any  $r,s\in[R_1,R_2]$ the following holds:
$$
|r-s|<\delta\quad\implies\quad U(r)\leq U(s)+L|r-s|.
$$  
Fix $\delta>0$ such that $0<R_1-\delta<R_2+\delta<R$ and pick $L>0$ such that
\begin{equation}\label{9febeq2}
L>\frac{2\left\|U\right\|_{L^\infty([R_1-\delta,R_2+\delta])}}{\delta}\quad\;\text{and}\;\quad L^p-\frac{L}{R_1-\delta}>C.
\end{equation}
Set
$$
\phi(r)=U(s)+L|r-s|,\quad\;r\in(s-\delta,s+\delta).
$$
The result follows by showing that $U\leq\phi$ in $r\in(s-\delta,s+\delta)$. Suppose by contradiction that there exists $r_0\in(s-\delta,s+\delta)$ such that 
 \begin{equation}\label{9febeq3}
U(r_0)>\phi(r_0).
\end{equation}
Since 
\begin{equation}\label{9febeq4}
\phi(s\pm\delta)>U(s\pm\delta)
\end{equation}
by the first condition in \eqref{9febeq2}, then \eqref{9febeq3}-\eqref{9febeq4} imply that
$$
\max_{[s-\delta,s+\delta]}(U-\phi)=(U-\phi)(r_m)
$$
for some $r_m\in(s-\delta,s+\delta)\backslash\left\{s\right\}$. Since $\phi$ is a smooth function for $r\neq s$, we have 
\begin{equation*}
\begin{split}
\phi'(r_m)&=L\frac{r_m-s}{|r_m-s|}\\
\frac{\phi'(r_m)}{r_m}&\leq \frac{L}{R_1-\delta}\\
\phi''(r_m)&=0.
\end{split}
\end{equation*}
Then, by  the subsolution property of $u$ at $|x|=r_m$, we obtain 
$$
 C\geq-\lambda_i(D^2u(x))+|Du(x)|^p\geq-\frac{L}{R_1-\delta}+L^p
$$
which is in contradiction to \eqref{9febeq2}.
\end{proof}

\begin{rem}\label{convexity}
\rm 
It is worth pointing out that, in the special case $i=1$, the local Lipschitz regularity of subsolution holds under more general assumptions than those of Proposition \ref{prop2regularity}. Any upper semicontinuous viscosity subsolution of
 \begin{equation}\label{eql1}-\lambda_1(D^2u)= C\,\quad\text{ in $\Omega$\, (convex subset of $\R^N$)}\end{equation} is, in fact,  a semiconvex function, hence locally Lipschitz in $\Omega$. For the convenience of the reader  we provide a proof of this fact, and we refer to \cite{Ob} for further properties concerning continuous subsolutions of the equation $-\lambda_1(D^2u)=0$ in the whole $\R^N$.

 In the following we use the notation $x=(x_1,x')\in\R^N$, with $x'\in\R^{N-1}$. Since $v=u+\frac C2|x|^2$ satisfies, in the viscosity sense, the inequality $-\lambda_1(D^2v)\leq0$ in $\Omega$, we can suppose $C=0$ in \eqref{eql1}.
Assume, by contradiction, that $v$ is not convex in $\Omega$. Hence there exist $\bar x,\bar z\in\Omega$ and $\theta\in(0,1)$ such that $v(\theta\bar x+(1-\theta)\bar z)>\theta v(\bar x)+(1-\theta)v(\bar z)$. Since the inequality $-\lambda_1(D^2v)\leq0$ is invariant by orthogonal transformations and translations,  without loss of generality we can assume that $\bar x=(-s,0')$, $\bar z=(t,0')$, with $0<s\leq t$, and that  $\theta\bar x+(1-\theta)\bar z=0$. Choose $a,b\in\mathbb R^N$ such that the function $l(x)=a\cdot x+b$ satisfies 
\begin{equation*}
v(\bar x)=l(\bar x)\;,\quad v(\bar z)=l(\bar z).
\end{equation*}
Since $(v-l)(0)>0$, we infer by semicontinuity  that for any positive $\varepsilon$ sufficiently small 
\begin{equation}\label{7lugl26eq2}
(v-l)(x)<\frac12(v-l)(0)\qquad\text{$\forall (x_1,x')\in\R^N$ s.t. $x_1=-s$ or $x_1=t$, and $|x'|<\varepsilon$}.
\end{equation}
Consider now the open cylinder 
$$
C_\varepsilon=\left\{(x_1,x')\in\mathbb R^N\,:\;x_1\in(-s,t)\,,\;|x'|<\varepsilon\right\}.
$$
It is clear that $C_\varepsilon\subset C_{\varepsilon_0}\subset\subset\Omega$ provided $\varepsilon<\varepsilon_0$ and $\varepsilon_0$ is small enough. Set
$$
\varphi(x)=-\varepsilon x_1^2+\frac{1}{\varepsilon^3}{|x'|}^2+l(x).
$$
We claim that, for small $\varepsilon$, the function $v-\varphi$ in $\overline {C_\varepsilon}$ attains its maximum at an interior point. For this it is sufficient to show that $(v-\varphi)(x)<(v-\varphi)(0)=(v-l)(0)$ for any $x\in\partial C_\varepsilon$. Indeed, if $x_1=-s$ or $x_1=t$ and $|x'|<\varepsilon$, by \eqref{7lugl26eq2} we have
$$
(v-\varphi)(x)<\frac12(v-l)(0)+\varepsilon t^2<(v-l)(0).
$$
If, instead, $x_1\in[-s,t]$ and $|x'|=\varepsilon$ we obtain
$$
(v-\varphi)(x)\leq \max_{x\in\overline {C_{\varepsilon_0}}}v(x)+\varepsilon t^2-\frac1\varepsilon+\max_{x\in\overline{C_{\varepsilon_0}}}(-l(x))<0.
$$
Thus 
$$
\max_{x\in\overline {C_{\varepsilon}}}(v-\varphi)(x)=(v-\varphi)(x_\varepsilon)
$$
for some $x_\varepsilon\in C_\varepsilon$. By the subsolution property of $v$ we have $-\lambda_1(D^2\varphi(x_\varepsilon))\leq0$. On the other hand, 
$D^2\varphi(x)=\text{diag}\,(-2\varepsilon,\frac{2}{\varepsilon^3},\ldots,\frac{2}{\varepsilon^3})$ for any $x\in\R^N$, hence $-\lambda_1(D^2\varphi(x_\varepsilon))=2\varepsilon>0$. This contradiction proves the convexity of $v$.

Another approach to demonstrating the convexity of a subsolution $v$ to $-\gl_1(D^2v) \leq 0$ 
in $\Omega$ relies on the equivalence between viscosity subsolutions and distributional subsolutions.  First, the inequality $-\gl_1(D^2v) \leq 0$ can be viewed as a collection of linear partial differential inequalities $-\du{D^2 v(x)\,\xi,\xi}\leq 0$ in $\Omega$ for all $\xi\in\R^N$. 
We begin by considering the case where $v \in C(\Omega)$. 
It is known (see \cite[Theorem 1]{I'}) that for any $\xi\in\R^N$, 
if $v\in C(\Omega)$ is a subsolution to $-\du{D^2v\,\xi,\xi}\leq 0$ (which is indeed the case), then it is also a distributional subsolution to the same inequality. 
As is well known--and easily verified--this collection of distributional inequalities for $v$ implies
its convexity.  In the general case, the convexity can be established by applying appropriate normalization to $v$, considering its sup-convolution,  and viewing $v$ as the pointwise limit thereof. The details of this argument are left to the reader.

%

\smallskip

 We conclude this remark by showing that, if instead $i>1$, then any H\"older continuity may fail at $x=0$ for $1<p\leq2$. Consider the following simple example: let  $u(x)=U(|x|)$ be the function
$$U(r)=
\left\{\begin{array}{cl}
-\frac{1}{\log r} & \text{if $0<r< \delta$}\\
0 & \text{if $r=0$.}
\end{array}\right.
$$
 It is clear that $U\notin C^{0,\alpha}(B_{\delta})$ for any $\alpha\in(0,1]$. On the other hand, by a straightforward computation, for $\delta$ small enough and any $0<r<\delta$ one has
$$
U''(r)<0<\frac{U'(r)}{r}
$$
and 
$$
-\lambda_i(D^2u(x))+|Du(x)|^p=\frac{1}{r^2\log^2r}\left(-1+\frac{r^{2-p}}{{\left(\log^2r\right)}^{p-1}}\right)\leq0. 
$$
Note in addition that there are no $C^2$ test functions touching $u$ from above at $x=0$, hence $u$ is in fact a viscosity subsolution of the above equation in the whole ball $B_{\delta}$.\\
As far as H\"older estimates for subsolution of degenerate elliptic equations with superquadratic Hamiltonians, we refere to \cite{AT,CDLP}.
\end{rem}

\subsection{A result for supersolutions}
\begin{pro}\label{prop4regularity}
Let $u\in\LSC(\Omega)$  be a locally bounded viscosity supersolution of 
\begin{equation}\label{eqprop4regularity}
-\lambda_N(D^2u)+|Du|^p=-C\quad\,\text{in\, $\Omega$},
\end{equation}
where $0<p\leq2$ and $C\geq0$. Then for any  $\omega\subset\subset\omega'\subset\subset\Omega$, there exists a positive constant  
{$L=L(p,C,\dist(\omega,\partial\omega'),\left\|u\right\|_{L^\infty(\omega')})$} such that 
\begin{equation}\label{LIP}
|u(x)-u(y)|\leq L|x-y| \quad\text{for $x,y\in \omega$.}
\end{equation}
\end{pro}
\begin{proof}
Let $\omega\subset\subset\omega'\subset\subset\Omega$. We claim that there  exist positive $\delta$ and  $L$, {depending on} $p$, $C$, {$\dist(\omega,\partial\omega')$}, $\left\|u\right\|_{L^\infty(\omega')}$, and there exists a function $\varphi\in C(\overline B_\delta)\cap C^2(B_\delta\backslash\left\{0\right\})$ such that 
\begin{equation}\label{10febeq1}
\varphi(0)=0\;,\qquad\varphi(x)\geq-L|x|\quad\forall x\in B_\delta\;, \qquad \varphi(x)\leq-2\left\|u\right\|_{L^\infty(\omega')}\quad\forall x\in\partial B_\delta
\end{equation}
and 
\begin{equation}\label{10febeq2}
-\lambda_N(D^2\varphi(x))+|D\varphi(x)|^p<-C\qquad\forall x\in B_\delta\backslash\left\{0\right\}.
\end{equation}
Conditions \eqref{10febeq1}-\eqref{10febeq2} readily imply the result. For this, pick any $y\in\omega$ and consider the map $x\in B_\delta(y)\mapsto u(x)-u(y)-\varphi(x-y)$. Reducing $\delta$ {if} necessary, we may further assume that $B_\delta(y)\subset\omega'$.  Such function vanishes at $y$ and it is nonnegative on $\partial B_\delta(y)$, because of \eqref{10febeq1}. Moreover it does not attain any local minimum in $B_\delta(y)\backslash\left\{y\right\}$ in view of the strict inequality in $\eqref{10febeq2}$ and the viscosity supersolution property of $u$. Hence $u(x)-u(y)-\varphi(x-y)\geq0$ for any $x\in  B_\delta(y)$ and, by the second property in \eqref{10febeq1}, we conclude
$$
u(x)-u(y)\geq-L|x-y|\qquad\forall x\in B_\delta(y).
$$
\smallskip

\noindent
Consider, for $x\in B_\delta$, the function $\varphi(x)=-\frac12\log\left(1+\delta^{-2}|x|\right)$, where $\delta$ is a  positive constant to be determined. \\
The first condition in \eqref{10febeq1} is satisfied regardless of $\delta$, the second holds with $L=\frac{\delta^{-2}}{2}$. Moreover, if $\delta$ is sufficiently small, we also have, for $x\in\partial B_\delta$, that
$$
\varphi(x)=-\frac12\log(1+\delta^{-1})\leq-2\left\|u\right\|_{L^\infty(\omega')}.
$$
Thus, all conditions in \eqref{10febeq1} are fulfilled. It remains to prove \eqref{10febeq2} for small $\delta$. Since $\varphi(x)=\phi(|x|)$, where $\phi(r)=-\frac12\log(1+\delta^{-2}r)$ satisfies 
{
$\phi''(r)=\frac{\delta^{-4}}{2{\left(1+\delta^{-2}r\right)}^2}>0>\frac{\phi'(r)}{r}$,
}%
 then for $x\in B_\delta\backslash\left\{0\right\}$
\begin{equation}\label{15apreq1}
-\lambda_N(D^2\varphi(x))+|D\varphi(x)|^p=-\frac{\delta^{-4}}{2{\left(1+\delta^{-2}|x|\right)}^2}\left(1-2^{1-p}{\left(\delta^2+|x|\right)}^{2-p}\right).
\end{equation}
Since $p\leq2$, we also have for $\delta$ sufficiently small 
\begin{equation}\label{15apreq2}
1-2^{1-p}{\left(\delta^2+|x|\right)}^{2-p}\geq1-2^{1-p}{\left(\delta^2+\delta\right)}^{2-p}\geq\frac12.
\end{equation}
Using \eqref{15apreq1}-\eqref{15apreq2}, we conclude that for any $x\in B_\delta\backslash\left\{0\right\}$
\begin{equation*}
-\lambda_N(D^2\varphi(x))+|D\varphi(x)|^p\leq-\frac{\delta^{-4}}{4{\left(1+\delta^{-2}|x|\right)}^2}\leq-\frac{\delta^{-2}}{4{\left(1+\delta\right)}^2}<-C
\end{equation*}
provided $\delta$ is small enough.
\end{proof}

\begin{rem}
\rm Proposition \ref{prop4regularity} is formulated for supersolutions of $-\lambda_N(D^2u)+|Du|^p=-C$ in $\Omega$, but its conclusion holds more generally for supersolutions of the equation 
\begin{equation}\label{3lug26eq1}
-\lambda_N(D^2u)+b|Du|^p=-C\quad\,\text{in\, $\Omega$},
\end{equation}
where $b$ is any real constant. Indeed, if $b\leq1$ then any supersolution  of \eqref{3lug26eq1}  is in turn a supersolution of \eqref{eqprop4regularity}. If instead $b>1$, it is sufficient to slightly modify the test function $\varphi$ used in the proof of Proposition \ref{prop4regularity}, by considering  $\varphi(x)=-\frac{1}{2b}\log\left(1+\delta^{-2}|x|\right)$, which yields the estimate \eqref{LIP}. In this case, the constant $L$  depends on $b$ as well.\\ It is also worth pointing out that from the interior Lipschitz estimates for supersolutions of \eqref{3lug26eq1}, it immediately follows that analogous estimates hold for subsolutions of $-\lambda_1(D^2v)+b|Dv|^p=C$ in  $\Omega$, simply by setting $v=-u$ and using the fact that $b$ is arbitrary. 
\end{rem}

The restriction on $p$ in Proposition \ref{prop4regularity} is optimal, as the following example shows.

\begin{exa}\label{ex1}
{\rm Let $u(x)=U(r)$ be the radial function
$$U(r)=
\left\{\begin{array}{cl}
\frac{1}{\log r} & \text{if $0<r< \delta$}\\
0 & \text{if $r=0$.}
\end{array}\right.
$$
In a neighborhood of the origin, the function $U(r)$ is convex and decreasing. Thus, if $p>2$ and $\delta=\delta(p)$ is sufficiently small, 
for $r=|x|\neq0$ and $r<\delta$ we have 
\begin{equation*}
\begin{split}
-\lambda_N(D^2u(x))+|Du(x)|^p&=-U''(r)+|U'(r)|^p\\
&=\frac{1}{r^p{\left(\log^2r\right)}^p}\left(1-r^{p-2}{\left(\log^2r\right)}^{p-1}+2r^{p-2}{\left|\log r\right|}^{2p-1}\right)\geq0.
\end{split}
\end{equation*}
Moreover there are no $C^2$ test {functions} touching $u$ at $x=0$ from below. Hence $u$ is a supersolution of 
$$
-\lambda_N(D^2u(x))+|Du(x)|^p=0	\;\quad\text{in $B_\delta$}
$$
and $u\notin C^{0,\alpha}(B_\delta)$ for any $\alpha>0$.}
\end{exa}

We wish to emphasize that the Hölder continuity property for supersolutions stated in Proposition \ref{prop4regularity} does not hold, in general, if the operator $\lambda_N(D^2u)$ is replaced by a uniformly elliptic operator $F(D^2u)$, and fails even for the smaller class of supersolutions to the equation $-F(D^2u) = -C$. Below we provide a counterexample in the special case $F(D^2u) = \Delta u$ and $C = 0$.

\begin{exa}
{\rm Let $u(x)=U(r)$ be the same radial function of Example \ref{ex1}.  The H\"older continuity fails at $x=0$ for any $\alpha>0$. On the other hand, $u$ is superharmonic in $B_1$. For this, note that there are no $\varphi\in C^2(B_1)$ such that $u-\varphi$ attains a local minimum at $x=0$, while for $x\in B_1\backslash\left\{0\right\}$ it holds
$$
-\Delta u(x)=\frac{1}{r^2\log^2r}\left(N-2-\frac{2}{\log r}\right)>0.
$$
}
\end{exa}

\subsection{Lipschitz estimates for solutions}

Consider the PDE
\begin{equation} \label{eq1Feb09}
-F(D^2u)+H(x, Du)+\gamma u=0 \ \quad\text{in $B_2$},
\end{equation}
where $\gamma\geq 0$ is a given constant. Here we focus on the case where
\[
F(X)=a\lambda_i(X),\qquad H(x,q)=b|q|^p-f(x),
\]
with constants $a,b>0$ and $p>1$.  
We assume that $f\in \Lip(\overline{B_2})$, with Lipschitz constant $\Lambda\geq 0$. That is, 
\[
|f(x)-f(y)|\leq\Lambda|x-y|.
\]
Set 
\[
\Lambda_0=\|f\|_\infty.
\]
Note
that for any $x,y\in B_1$, $q,q_i\in\R^N$, $i=1,2$,
\begin{align}
&H(x,q)\geq b|q|^p-\Lambda_0, \label{H1}
\\&|H(x,q)-H(y,q)|=|f(x)-f(y)|\leq \Lambda|x-y|, \label{H2}
\\ \label{H3} & |H(x,q_1)-H(x,q_2)|\leq bp (|q_1|^{p-1}+|q_2|^{p-1})|q_1-q_2|. 
\end{align}

\begin{pro}\label{prop1regularity}
 If $u\in C(\overline{B_2})$ is a viscosity solution of \eqref{eq1Feb09}, then
\[
|u(x)-u(y)|\leq L_0|x-y| \quad\text{for $ x,y\in B_{1/2}$},
\]
where 
the constant $L_0$ depends only on 
$\gamma\|u\|_\infty, p, a, b, \gL_0, \gL$. 
\end{pro}

While there is a vast body of research on Lipschitz continuity of solutions to elliptic equations, we will mention only \cite{AT, B1,B2, CDLP} here as highly relevant prior studies in the viscous solution framework.
The proof below follows that of \cite[Theorem 3.1]{AT}, with minor variations, which mostly take care of the effects of nonlinearity of the operator $F(D^2 u)$. 

\begin{rem}
We point out that the Lipschitz estimate in Proposition \ref{prop1regularity} is quite general. Indeed, it holds not only for equations whose principal part is given by $F(X)=a\lambda_i(X)$, with $a>0$ and $i\in\left\{1,\ldots,N\right\}$, but for any operator $F$ satisfying  assumptions \eqref{F0}-\eqref{F1}-\eqref{F2}.
\end{rem}

\begin{proof}

According to \cite[Theorem 3.1]{AT}, there is a function $g\in C^2(B_1)$ and a 
constant $C>0$ such that
\begin{align}
&g=1 \ \ \IN B_{1/2}, \label{phi0}
\\
&g\geq 1 \ \ \IN B_1, \label{phi1}
\\&\lim_{|x|\to 1^-}g(x)=\infty, \label{phi2}
\\&|Dg(x)|\leq Cg(x)^p \ \ \FOR x\in B_1, \label{phi3}
\\&|D^2g(x)|\leq Cg(x)^{2p-1} \ \ \FOR x\in B_1.\label{phi4}
\end{align}

Let $u\in C(\ol{B_2})$ be a solution of \eqref{eq1Feb09}. 
Let $L_0>0$ be a constant, which will be fixed so as to 
depend only on $\gamma\|u\|_\infty, p, a, b, \gL_0, \gL, C$
at the end of the proof.  

If 
\[
u(x)-u(y)-Lg(y)|x-y|\leq 0 \ \ \ \FORALL  (x,y)\in B_{1/2}\tim B_{1/2},
\]
then 
\[
u(x)-u(y)\leq L|x-y| \ \ \FORALL  x,y\in B_{1/2},
\]
and, in addition, if $L$ can be chosen so that $L\leq L_0$, then we are done. 

Thus, we assume that 
\begin{equation}\label{sup>0}
\sup_{x,y\in B_{1/2}}(u(x)-u(y)-L|x-y|)=\sup_{x,y\in B_{1/2}}(u(x)-u(y)-Lg(y)|x-y|)> 0,
\end{equation}
and will show that this leads a contradiction if $L>L_0$, with an appropriate choice of constant $L_0$. 

Observe that, for every fixed $\delta>0$, if 
\[
u(x)-u(y)\leq L|x-y| \ \ \FORALL x,y\in B_{1/2}, \text{ with } |x-y|\leq \delta,
\]
then $u(x)-u(y)\leq L|x-y|$ for all $x,y\in B_{1/2}$. 
By \eqref{sup>0},
there are a constant $c>0$ and points $x,y\in B_{1/2}$ such that 
\[
c<u(x)-u(y)-L|x-y|.
\]
Let $n\in\N$ and consider a partition of the line segment $[x,y]$ into $n$ equal segments. 
As easily seen, we can choose one of the $n$ equal segments, denoted by $[x_n,y_n]$, so that  
$n^{-1}c<u(x_n)-u(y_n)-L|x_n-y_n|$. Let $\ga>0$. Noting that $|x_n-y_n|=n^{-1}|x-y|<1$, we deduce that $
u(x_n)-u(y_n)-L|x_n-y_n|-\fr{\ga}{2}|x_n-y_n|^2 
>n^{-1}\left(c-\fr{\ga}{2n}\right). $ 
This shows that there exists a point $(\bar x, \bar y)\in B_{1/2}\tim B_{1/2}$ such that 
\[
u(\bar x)-u(\bar y)-Lg(\bar y)|\bar x-\bar y|-\frac{\ga}{2}|\bar x-\bar y|^2>0. 
\] 
%
%
In what follows, we consider the function 
\[
\Phi(x,y):=Lg(y)|x-y|+\frac{\ga}{2}|x-y|^2 \ \ \FOR (x,y)\in \ol{B_2}\tim B_1.
\]
where $L, \ga$ are constants satisfying $L,\ga\geq 1$.  
It follows that 
\[
0<\sup_{(x,y)\in\ol{B_2}\tim B_1}(u(x)-u(y)-\Phi(x,y))\leq \max_{x,y\in\ol{B_2}}(u(x)-u(y))<+\infty,
\]
and the maximum of $u(x)-u(y)-\Phi(x,y)$ is achieved at a point in $\ol{B_2}\tim B_1$. The attainability of this maximum value is the only place in the proof where the continuity of $u$ plays a role.

Let $(\hat x, \hat y)\in \ol{B_2}\tim B_1$ be a maximum point of
the function $u(x)-u(y)-\Phi(x,y)$. Clearly, we have $\x\not= \y$.

In addition that $|\hat x-\hat y|>0$, since  
\[
u(\hat x)-u(\hat y)>\Phi(\hat x,\hat y)\geq \frac{\ga}{2}|\hat x-\hat y|^2,
\]
by selecting $\ga>1$ large enough, 
we may always assume that 
\begin{equation}\label{dxy}
|\hat x-\hat y|\leq 1. 
\end{equation}
Moreover, since $\hat y\in B_1$, we have 
\[
\hat x\in B_2.
\] 
As the choice of $L_0$ will not depend on $\ga$, we can choose $\ga$ to be as large as required in the proof. 

Since $u$ is a viscosity solution of \eqref{eq1Feb09}, we have
\begin{equation}\label{visco}\begin{cases}
-F(X)+H(\hat x, D_x\Phi(\hat x,\hat y))+\gamma u(\hat x)\leq 0&\text{ if } \ (X,D_x\Phi(\hat x,\hat y))\in \ol{J}^{2,+}u(\hat x),\\[3pt] 
-F(-Y)+H(\hat y,-D_y\Phi(\hat x,\hat y))+\gamma u(\hat y)\geq 0&\text{ if } \  (Y,D_y\Phi(\hat x,\hat y))\in\ol{J}^{2,+}(-u)(\hat y).
\end{cases}
\end{equation}

We will let  $r=|x-y|$.
Recall that 
$$r_x=D_xr=\frac{x-y}{r}=-r_y=-D_yr,$$
while
$$r_{xx}=D^2_{xx}r=\frac{1}{r}(I_N-r_x\otimes r_x)=r_{yy}=-r_{xy}=-r_{yx},$$
where $I_N$ denotes the identity matrix of order $N$.  Here and henceforth, 
we write $r_x=D_x r$, $r_y=D_y r$, $r_{xx}=D_x^2 r$, etc. 
We compute that, if $x\not=y$ and $y\in B_1$, 
\beq \label{eq3-2Isa}\bald
\Phi_x(x,y)&\,=(Lg(y)+\ga r)r_x,\\
\Phi_{xx}(x,y)&\,=(Lg(y)+\ga r)r_{xx}+\ga r_x\otimes r_x\\
&\,=(\frac{Lg(y)}{r}+\ga )I_N-\frac{Lg(y)}{r} r_x\otimes r_x
\\&\leq\frac{|\Phi_x(x,y)|}{r}I_N.
\eald
\eeq

We have obtained that for $\hat r=|\hat x-\hat y|$ and $
\hat\rho=Lg(\hat y)+\ga\hat r$,
the following holds
\[
|\Phi_x(\hat x,\hat y)|= \hat\rho>L\geq 1\quad
\AND\quad \Phi_{xx}(x,y)\leq \frac{\hat\rho}{\hat r} I_N.
\]

Observe also that if $x\not=y$, 
\beq\label{eq3-1}\bald
\Phi_y(x,y)\,&=(Lg(y)+\ga r)r_y + LDg(y)r=- \Phi_x(x,y)+ LDg(y)r
\\ \Phi_{xy}(x,y)\,&=- \Phi_{xx}(x,y)+ r_x\otimes LDg(y)
\\ \Phi_{yx}(x,y)\,&=- \Phi_{xx}(x,y)+ LDg(y)\otimes r_x
\\ \Phi_{yy}(x,y)\,&=(Lg(y)+\ga r)r_{yy} +\ga r_y\ot r_y+ LD^2g(y)r+ L(Dg(y)\ot r_y
+r_y\ot Dg(y))
\\&= \Phi_{xx}(x,y) + LD^2g(y)r+ L(Dg(y)\ot r_y+r_y\ot Dg(y)), 
\eald \eeq
and, by using \eqref{phi3},
\beq\label{eq3}\bald
&|\Phi_x(\hat x,\hat y)+\Phi_y(\hat x,\hat y)|=L|Dg(\hat y)|\hat r
\leq LCg(\hat y)^p\hat r, \quad
\\&|\Phi_y(\hat x,\hat y)|\leq \hat\rho+LCg(\hat y)^p\hat r. 
\eald
\eeq

Hence, by \eqref{visco}, \eqref{H1}, \eqref{F1}, \eqref{F2}, and \eqref{F0}, we get 
\[\bald
0&\,\geq -F(\Phi_{xx}(\hat x,\hat y))+H(\hat x,\Phi_x(\hat x,\hat y))+\gamma u(\hat x)
\\&\,
\geq -F\left(\frac{\hat\rho}{\hat r}I_N\right)+H(\hat x,\Phi_x(\hat x,\hat y))+\gamma u(\hat x)
\\&\geq -a\frac{\hat\rho}{\hat r} +b \hat\rho^p -\Lambda_0+\gamma u(\hat x).
\eald\]
Since $\hat r\leq 1$, we find that
\begin{equation}\label{eq3+} 
b\hat\rho^p \hat r \leq a \hat\rho + (\Lambda_0+\gamma\|u\|_\infty)\hat r. 
\end{equation}

\bigskip
%

Let $\varepsilon>0$. We choose $X_\varepsilon,Y_\varepsilon\in\cS(N)$ (see, e.g.,  \cite[Theorem 3.2]{CIL}) so that 
\[\bald
(\Phi_x(\hat x,\hat y), \,X_\varepsilon)&\in \ol J^{2,+} u(\hat x), \quad (\Phi_y(\hat x,\hat y),\,Y_\varepsilon)\in\ol J^{2,+}(-u)(\hat y),
\\& \bmat X_\varepsilon & 0\\0& Y_\varepsilon \emat\leq  A+\varepsilon A^2,
\eald\]
where $A:=D^2\Phi(\hat x,\hat y)\in \cS(2N)$. Immediate consequences of these are
\begin{equation}\label{eq4}\bald
-F(X_\varepsilon)+H(\hat x, \Phi_x(\hat x,\hat y))+\gamma u(\hat x)&\,\leq 0,
\\ -F(-Y_\varepsilon)+H(\hat y,  -\Phi_y(\hat x,\hat y))+\gamma u(\hat y)&\,\geq 0.
\eald \end{equation}
Let $s>1$ be a constant to be fixed later. By \eqref{eq4} together with \eqref{F0}, we have
\begin{equation}\label{eq5}
0\geq -F(s^2X_\varepsilon)+s^2H(\hat x,\Phi_x(\hat x,\hat y)) 
+F(-Y_\varepsilon)-H(\hat y,-\Phi_y(\hat x,\hat y))+s^2\gamma u(\hat x)-\gamma u(\hat y).
\end{equation}
Observe that 
\begin{equation}\label{eq6} \bald
s^2 X_\varepsilon+Y_\varepsilon&\,=
\bmat sI_N & I_N\emat \bmat X_\varepsilon&0 \\ 0 &Y_\varepsilon\emat \bmat sI_N \\ I_N\emat
\leq \bmat sI_N & I_N\emat (A+\varepsilon A^2) \bmat sI_N \\ I_N\emat
\\ &\,=\bmat sI_N & I_N\emat A\bmat sI_N \\ I_N\emat+O(\varepsilon)\\
&\, = s^2\Phi_{xx}(\hat x,\hat y)+s(\Phi_{xy}(\hat x,\hat y)+\Phi_{yx}(\hat x,\hat y))+\Phi_{yy}(\hat x,\hat y)+ O(\varepsilon).
\eald\end{equation}
Using \eqref{eq3-2Isa}-\eqref{eq3-1}, we get
\begin{equation}\label{eq7isa} \bald
&s^2\Phi_{xx}(\hat x,\hat y)+s(\Phi_{xy}(\hat x,\hat y)+\Phi_{yx}(\hat x,\hat y))+\Phi_{yy}(\hat x,\hat y)\\
=&(s-1)^2\Phi_{xx}(\hat x,\hat y) + (s-1)L( r_x\otimes Dg(\hat y)+ Dg(\hat y)\otimes r_x) +LD^2g(\hat y) \hat r\\
\leq &(s-1)^2\frac{\hat \rho}{\hat r}I_N+2(s-1)L|Dg(\hat y)|I_N+L|D^2g(\hat y) |\hat r I_N.
\eald\end{equation}

Finally we have obtained, using \eqref{phi3},\eqref{phi4}
\begin{equation}\label{eq7isa} \bald
s^2 X_\varepsilon+Y_\varepsilon&\,\leq  \left(\frac{(s-1)^2\hat\rho}{\hat r}
+2(s-1)LCg(\hat y)^p
+LCg(\hat y)^{2p-1}\hat r\right)I_N + O(\varepsilon).
\eald\end{equation}

%
%
We denote by $\hat g=g(\hat y)$ and we combine this with \eqref{eq5}, \eqref{F1}, \eqref{F2}, and  also the fact that $u(\hat x)>u(\hat y)$, 
to deduce 
\[\bald
0&\,\geq -F\left(-Y_\varepsilon +(s-1)^2 \frac{\hr}{\hat r}I_N 
+2(s-1)LC \hat g ^p I_N +LC\hat g ^{2p-1}\hat r I_N+O(\varepsilon)\right)
\\ &\,\ \quad+F(-Y_\varepsilon) +s^2H(\hat x,\Phi_x(\hat x,\hat y))-H(\hat y,-\Phi_y(\hat x,y))+s^2\gamma u(\hat x)-\gamma u(\hat y)
\\&\,\geq -a\left((s-1)^2 \frac{\hr}{\hat r} 
+2(s-1)CL\hat g ^p+CL\hat g ^{2p-1} \hat r+O(\varepsilon)\right)
\\ &\,\ \quad+s^2H(\hat x,\Phi_x(\hat x,\hat y))-H(\hat y,-\Phi_y(\hat x,y))+(s^2-1)\gamma u(\hat y).
\eald\]
Sending $\varepsilon\to 0$ yields
\[
\bald
0&\,\geq -a\left((s-1)^2 \frac{\hr}{\hat r} 
+2(s-1)LC \hat g ^p +LC\hat g ^{2p-1}\hat r\right)
\\ &\,\ \quad+s^2H(\hat x,\Phi_x(\hat x,\hat y))-H(\hat y,-\Phi_y(\hat x,\hat y))
+(s^2-1)\gamma u(\hat y).
\eald
\]
We choose
\[
s=1+\beta \hat r,
\]
where $\beta>0$ is still to be selected, to obtain
\begin{equation}\label{eq9}\bald
0&\,\geq -a\left(\beta^2\hr
+2\beta LC\hat g ^p  
+LC\hat g ^{2p-1}\right)\hat r
\\ &\,\ \quad+(1+\beta\hat r)^2H(\hat x,\Phi_x(\hat x,\hat y))-H(\hat y,-\Phi_y(\hat x,y))+\beta\hat r (2+\beta\hat r)\gamma u(\hat y).
\eald
\end{equation}
Recalling \eqref{eq3} and 
using \eqref{H1}, \eqref{H2}, and \eqref{H3}, we deduce that 
\[\bald
(1+&\beta\hat r)^2H(\hat x,\Phi_x(\hat x,\hat y))-H(\hat y,-\Phi_y(\hat x,y))
\\&\,=[(1+\beta\hat r)^2-1]H(\hat x,\Phi_x(\hat x,\hat y))
+H(\hat x,\Phi_x(\hat x,\hat y))-H(\hat y,\Phi_x(\hat x,\hat y))
\\ & \ \quad 
+H(\hat y,\Phi_x(\hat x,\hat y))-H(\hat y,-\Phi_y(\hat x,\hat y))
\\&\,\geq \beta(2+\beta\hat r)\hat r(b \hr^p-\Lambda_0) 
- \Lambda\hat r
-2bp\left(\hr+LC\hat g ^p\hat r\right)^{p-1}
LC\hat g ^p\hat r.
\eald\] 
Inserting the inequality above into \eqref{eq9} yields
\[\bald
0&\,\geq -a\left(\beta^2\hr
+2\beta LC\hat g ^p  
+LC\hat g ^{2p-1}\right)\hat r
+\beta(2+\beta\hat r)\hat r(b \hr^p-\Lambda_0+\gamma u(\y)) 
- \Lambda\hat r
\\ & \ \quad 
-2bp\left(\hr+LC\hat g ^p\hat r\right)^{p-1}
LC\hat g ^p\hat r.
\eald
\]
We assume henceforth by selecting $L$ large enough that 
\begin{equation} \label{L1}
L\geq \left(\frac{2(\Lambda_0+\gamma\|u\|_\infty)}{b} \right)^{\frac 1p},
\end{equation}
which, in particular, implies that 
$\,
b\hr^p-\Lambda_0 + \gamma u(\y)\geq \frac {b} 2\hr^p. \,$ 
Hence, 
\[\bald
0&\,\geq -a\left(\beta^2\hr
+2\beta LC\hat g ^p  
+LC\hat g ^{2p-1}\right)\hat r
+\beta \hat r b \hr^p 
- \Lambda\hat r
\\ & \ \quad 
-2bp\left(\hr+LC\hat g ^p\hat r\right)^{p-1}
LC\hat g ^p\hat r.
\eald
\]
Dividing this by $\hat r$, we get  
\begin{equation}\label{eq10}\bald
0&\,\geq -a\left(\beta^2\hr
+2\beta LC\hat g ^p  
+LC\hat g ^{2p-1}\right)
\\&\,\quad
+\beta b \hr^p 
- \Lambda
-2bp (\hr +LC\hat g ^p\hat r)^{p-1}
LC\hat g ^p.
\eald \end{equation}

Combining \eqref{eq3+}  
and 
\eqref{L1}, we infer that 
\[
b\hat\rho^p \hat r \leq a\hr+
\frac 12
bL^p\hat r
\leq a\hr+\frac 12 b\hr^p\hat r,
\]
and moreover, 
\[
\hr^p\hat r\leq \frac{2a}{b}\hr. 
\]
Since $L\hat g>\hat\rho$, we thus deduce that 
\beq\label{eq10+}
LC\hat g ^p\hat r
=L^{1-p}C(L\hat g )^p\hat r
\leq L^{1-p}C\hr^p\hat r
\leq L^{1-p}C\frac{2a}{b} \hr. 
\eeq
We assume that 
\begin{equation}\label{L2}
L\geq \left(\frac{2Ca}{b}\right)^{\frac{1}{p-1}}.
\end{equation}
Observe by this and \eqref{eq10+} that 
 $\,
 LC\hat g ^p\hat r\leq \hr
 \ $
 and 
\[\bald
2bp\left(\hr+LC\hat g ^p\hat r\right)^{p-1}
LC\hat g ^p
\leq 2bp(2\hr)^{p-1}
L^{1-p}C(L\hat g )^p
\leq 2^p bp L^{1-p}C\hr^{2p-1} 
\eald
\]
Now, from \eqref{eq10}, we obtain 
\[\bald
0&\,\geq -a\left(\beta^2\hr
+2\beta LC\hat g ^p  
+LC\hat g ^{2p-1}\right)
+\beta b \hr^p 
- \Lambda
-2^p bp L^{1-p}C\hr ^{2p-1}.
\eald
\]
Setting $\beta=\gamma \hr^{p-1}$, with $\gamma>0$,  and using the inequalities $1\leq \hat g  \leq L^{-1}\hr$, we find that 
\[\bald
0&\,\geq -a\left(\gamma^2\hr^{2p-1}
+2\gamma LC\hr^{p-1}\hat g ^p  
+LC\hat g ^{2p-1}\right)
\\&\,\quad
+\gamma b \hr^{2p-1} 
- \Lambda
-2^p bp L^{1-p}C\hr ^{2p-1}
\\&
\geq\hr^{2p-1}\left(\gamma b-a\left(\gamma^2+2\gamma CL^{1-p}+CL^{2(1-p)}\right)
-\Lambda L^{1-2p}-2^p bp CL^{1-p}\right).
\eald
\]
We select 
$\, 
\gamma=\frac{b}{2a},
\ $ 
to obtain 
\[\bald
0&\,
\geq \hr^{2p-1}\left(\frac{b^2}{4a} -b CL^{1-p}-aCL^{2(1-p)}
-\Lambda L^{1-2p}-2^p bp CL^{1-p}\right)
\\
&\,
\geq \hr^{2p-1}\left(\frac{b^2}{4a} -b CL^{1-p}-aCL^{1-p}
-\Lambda L^{1-p}-2^p bp CL^{1-p}\right)
\\
&\,
\geq \hr^{2p-1}\left(\frac{b^2}{4a} -L^{1-p}\left (b C+aC
+\Lambda+2^p bpC\right)\right).
\eald
\]
 
We assume that 
\begin{equation}\label{L3}
L^{p-1}\geq \frac{5a \left(b C+aC
+\Lambda+2^p bpC\right)}{b^2},
\end{equation}
which yields a contradiction:  
\beq\label{contra}
0\geq \hr^{2p-1} \Big( 
\frac{b^2}{4a} 
 - L^{1-p}\left(b C+aC
+\Lambda+2^p bpC\right)\Big)>0.
\eeq

In light of  \eqref{L1}, \eqref{L2}, and \eqref{L3}, we define 
\[
L_0:=\max\Big\{1,\left(\frac{2(\Lambda_0+\gamma\|u\|_\infty)}{b} \right)^{\frac{1}{p-1}},\left(\frac{2Ca}{b}\right)^{\frac 1p},
\left(\frac{5a \left(b C+aC
+\Lambda+2^p bpC\right)}{b^2}\right)^{\frac {1}{p-1}}
\Big\},
\]
and conclude that if $L\geq L_0$, then we arrive at 
\eqref{contra}, a contradiction,  starting with \eqref{sup>0}, which completes the proof.  
Note that $L_0$ depends only on $\gamma\|u\|_\infty, p, a, b, \gL_0, \gL, C$.
\end{proof}

\section{On the existence of radial solutions}\label{radial}

 Given $f\in C((0,R))$, with $R>0$, we first give a necessary condition for the existence of subsolutions to the equation 
\begin{equation}\label{4mareq1}
-\lambda_i(D^2u)+{|Du|}^p=f(|x|)\quad\;\text{in $B_R\backslash\left\{0\right\}$}.
\end{equation}
For later purpose, given $p>1$, let us introduce the function
\begin{equation}\label{9mareq4}
\gamma(r,z)=-\frac zr+|z|^p\quad\;\text{for $r>0$ and $z\in\mathbb R$}.
\end{equation}
Note that, for any fixed $r>0$, one has
\begin{equation}\label{9mareq5}
\min_{z\in\mathbb R}\gamma(r,z)=\gamma(r,z_{\min})=-\frac{p-1}{\left(pr\right)^{\frac{p}{p-1}}}
\end{equation}
where $z_{\min}=\left(\frac{1}{pr}\right)^{\frac{1}{p-1}}$.
\begin{pro}\label{4marprop1}
 Let $f\in C((0,R))$. If $u\in \USC(B_R\backslash\left\{0\right\})$ is a viscosity subsolution of \eqref{4mareq1} and $i<N$, then
\begin{equation}\label{4mareq2}
-\frac{p-1}{(pr)^{\frac{p}{p-1}}}\leq f(r) \ \ \text{for $r\in(0,R)$.}
\end{equation}
\end{pro}
\begin{proof}
Starting from $u$ and using the fact that the inequality \eqref{4mareq1} is invariant by rotations, we claim that there exists a radial subsolution of \eqref{4mareq1}. For this, since $u\in\USC(B_R\backslash\left\{0\right\})$, the radial function 
$$
w(x)=\max_{|y|=|x|}u(y)\,,\;\quad x\in B_R\backslash\left\{0\right\}
$$
is well defined and $w(x)<+\infty$ for any $x\in B_R\backslash\left\{0\right\}$. Moreover it is upper semicontinuous and thus it coincides with its upper semicontinuous envelope $w^*$. For this, consider any $x_0\in  B_R\backslash\left\{0\right\}$ and take any sequence $x_n\in  B_R\backslash\left\{0\right\}$  such that $x_n\to x_0$ as $n\to+\infty$. By definition of $w$, there exists a sequence $y_n\in  B_R\backslash\left\{0\right\}$ such that $w(x_n)=u(y_n)$ and $|y_n|=|x_n|$. We can extract a subsequence $y_{n_k}\in  B_R\backslash\left\{0\right\}$ such that 
\begin{equation}\label{9mareq1}
\limsup_{n\to+\infty}u(y_n)=\lim_{k\to+\infty}u(y_{n_k})
\end{equation}
and 
\begin{equation}\label{9mareq2}
\lim_{k\to+\infty}y_{n_k}=y_0,
\end{equation}
for some $y_0\in B_R\backslash\left\{0\right\}$ with $|y_0|=|x_0|$. Using \eqref{9mareq1}-\eqref{9mareq2} and $u\in \USC(B_R\backslash\left\{0\right\})$, we obtain
$$
\limsup_{n\to+\infty}w(x_n)=\lim_{k\to+\infty}u(y_{n_k})\leq u(y_0)\leq\max_{|y|=|x_0|}u(y)=w(x_0).
$$
To show that $w$ is a subsolution, we  observe  that the  function $w$ can be equivalently written as 
$$
w(x)=\max\left\{u(Ox)\,:\;\text{$O$ orthogonal matrix}\right\}
$$
and that $u(Ox)$ is still a subsolution \eqref{4mareq1}, for any orthogonal matrix $O$. Then, in view of \cite[Lemma 4.2]{CIL}, $w$ is in turn a radial subsolution of \eqref{4mareq1} as claimed.

\smallskip
Set $w(x)= W(|x|)$. For any $r_0\in(0,R)$ and  $\Phi\in C^2((0,R))$ such that $W-\Phi$ has a local maximum at $r_0$, we infer that the function $u(x)-\phi(x)$ with $\phi(x)=\Phi(|x|)$ attains a local maximum at any $x_0$ with $|x_0|=r_0$. Given that the eigenvalues of $D^2\phi(x_0)$ are $\frac{\Phi'(r_0)}{r_0}$ (with multiplicity, at least, $N-1$) and $\Phi''(r_0)$, then 
\begin{equation}\label{9mareq3}
\lambda_i(D^2\phi(x_0))\leq \frac{\Phi'(r_0)}{r_0}
\end{equation}
since $i<N$. By \eqref{4mareq1}-\eqref{9mareq4}-\eqref{9mareq5} and \eqref{9mareq3} we obtain
\begin{equation}\label{9mareq6}
f(r_0)\geq-\frac{\Phi'(r_0)}{r_0}+{|\Phi'(r_0)|}^p\geq\min_{z\in\mathbb R}\gamma(r_0,z)=-\frac{p-1}{(pr_0)^{\frac{p}{p-1}}}.
\end{equation}
Since the set $\left\{r_0\in(0,R)\,:\;J^{2,+}W(r_0)\neq\emptyset\right\}$ is dense in $(0,R)$, then using \eqref{9mareq6} and the continuity of $f$, we obtain the thesis \eqref{4mareq2}.
\end{proof}

\begin{rem}
\rm In the special case $f(x)=-C$, a negative constant function, Proposition \ref{4marprop1} yields a nonexistence result when $\displaystyle R>\frac{1}{p}{\left(\frac{p-1}{C}\right)}^\frac{p-1}{p}$, complementing the result obtained in \cite[Proposition 3.1]{BGR}.
\end{rem}

\begin{rem}
\rm
%
The result of Proposition \ref{4marprop1} can be easily generalized to subsolutions of equations involving
  weighted partial sums of the eigenvalues $\lambda_1,\ldots,\lambda_{N-1}$:
	\begin{equation}\label{12mareq1}
	-\sum_{i=1}^{N-1}a_i\lambda_i(D^2u)+{|Du|}^p=f(|x|)\quad\;\text{in $B_R\backslash\left\{0\right\}$},
	\end{equation}
 $a_i$ being nonnegative real numbers for any $i=1,\ldots,N-1$.  Accordingly, condition \eqref{4mareq2} is replaced by the more general
\begin{equation}\label{12mareq2}
-(p-1){\left(\frac{\sum_{i=1}^{N-1}a_i}{pr}\right)}^{\frac{p}{p-1}}\leq f(r) \ \ \text{for $r\in(0,R)$.}
\end{equation}
\end{rem}

\medskip
 
The next two Propositions \ref{4marprop2} and \ref{propLN} establish the existence of $C^2$-radial solutions. The former treats the case $i\leq N-1$ and the latter the case $i=N$ in equation \eqref{4mareq1}

\begin{pro}\label{4marprop2}
Assume $f\in C^1([0,R])$ and 
\begin{equation}\label{9mareq11}
-\frac{p-1}{(pr)^{\frac{p}{p-1}}}< f(r) \ \ \text{for $r\in(0,R)$.}
\end{equation}
Then:
\begin{enumerate}
	\item[(a)] if $1<i<N$, there exists {a radial solution $u\in C^2(\overline B_R)$} of 
	\begin{equation}\label{9mareq7}
	-\lambda_i(D^2u)+|Du|^p=f(|x|)\;\quad\text{in $B_R$};
	\end{equation}
	\item[(b)] if $i=1$ and we further assume that $$f(r)\leq0,\;f'(r)\leq0\quad\text{in $[0,R]$},$$
	then there exists a radially nondecreasing solution $u\in C^2(\overline B_R)$ of \eqref{9mareq7};
	\item[(c)] if $i=1$ and we further assume that $f$ is a nonnegative constant function,
	then there exists a radially nonincreasing solution $u\in C^2(\overline B_R)$ of \eqref{9mareq7}.
\end{enumerate}
\end{pro}
\begin{proof}
Using the notations  \eqref{9mareq4}-\eqref{9mareq5}, given any $r\in(0,R]$, there exists a unique real number $s(r)$ such that 
\begin{equation}\label{10mareq2}
-r f(r)\leq s(r)<z_{\min}=\left(\frac 1{pr}\right)^{\frac 1{p-1}}
\end{equation}
 and 
\begin{equation}\label{9mareq9}
\gamma(r,s(r))=f(r).
\end{equation}
Since $\gamma_z(r,z)<0$ for $z<z_{\min}$, by the implicit function theorem we infer that $s\in C^1((0,R])$ and 
$$
\gamma_z(r,s(r))<0\quad\,\text{for $r\in(0,R]$}.
$$
Moreover, for any fixed $\alpha\in(0,1)$, we have
$$
\lim_{r\to 0^+}\gamma(r,r^\alpha)=\lim_{r\to 0^+}\left(-\frac{1}{r^{1-\alpha}}+r^{\alpha p}\right)=-\infty.
$$
Thus,  for sufficiently small $r$, it turns out that $\displaystyle\gamma(r,r^\alpha)<\min_{[0,R]}f$ and 
$$
-r f(r)\leq s(r)<r^\alpha.
$$
It follows that
\begin{equation}\label{9mareq8}
\lim_{r\to 0^+}s(r)=0.
\end{equation}
Differentiating $\gamma(r,s(r))=f(r)$ yields
\begin{equation*}
\gamma_r(r,s(r))+\gamma_z(r,s(r))s'(r)=f'(r),
\end{equation*}
which reads
\begin{equation}\label{9mareq10bis}
\frac{s(r)}{r}=r\big(f'(r) -\gamma_z(r,s(r))s'(r)\big) 
=rf'(r)+\big(1-rp|s(r)|^{p-2}s(r))\big)s'(r). 
\end{equation}
Using \eqref{9mareq9}-\eqref{9mareq8}-\eqref{9mareq10bis} we obtain 
\[
s(r)=r(-f(r)+|s(r)|^p)=-f(0)r+o(r) \ \ \text{ as \ } r\to 0^+,
\]
and
\begin{equation*}
\begin{split}
s'(r)&=\frac{\frac{s(r)}{r}-rf'(r)}{1-r p|s(r)|^{p-2}s(r)}\\
&=\frac{-f(0)+o(1)-rf'(r)}{1-r p|s(r)|^{p-2}s(r)}=-f(0)+o(1) \ \ \text{ as }r\to 0^+.
\end{split}
\end{equation*}
Then $s\in C^1\left([0,R]\right)$ and $s'(0)=-f(0)$. Define 
\begin{equation*}
U(r)=\int_0^r s(t)dt \ \ \FOR r\in[0,R]
\end{equation*}
and 
\begin{equation*}
 u(x)=U(|x|).
\end{equation*}
Then $U\in C^2([0,R])$ and   
\[
 \gamma(r,U'(r))=f(r)\ \ \FOR r\in (0,R]. 
\]
Moreover, since $U'(0)=s(0)=0$ and $U''(0)=s'(0)=-f(0)$, we find that 
\begin{equation}\label{10mareq2*}
D^2u(0)=-f(0)I_N
\end{equation}
and that $u\in C^2(\overline B_R)$.

\medskip
\noindent
Case (a).  We have $$\lambda_i(D^2u(x))=\frac{U'(|x|)}{|x|}\ \ \FOR x\in B_R\backslash\left\{0\right\}$$
and $$
-\lambda_i(D^2u)+|Du|^p=f(|x|) \ \ \IN B_R\backslash\left\{0\right\}.
$$
 Using \eqref{10mareq2*}, we also have
\[
-\lambda_i(D^2u(0))+|Du(0)|^p=f(0).
\]  
Hence, $u$ is a classical solution of \eqref{9mareq7}.

\medskip
\noindent
Case (b). Using \eqref{10mareq2} and the assumption $f(r)\leq0$, we infer that  $0\leq s(r)<z_{\min}$. In view of \eqref{9mareq10bis}, by the assumption $f'(r)\leq0$, we also have
\[
0\leq \frac{s(r)}{r}\leq -\underbrace{r\gamma_z(r,s(r))}_{<0}s'(r),
\]
and $s'(r)\geq 0$.  Moreover, using again \eqref{9mareq10bis}, we obtain 
\begin{equation}\label{f=constant}
\frac{s(r)}{r}\leq s'(r).
\end{equation}
This means that $0\leq\frac{U'(r)}{r}\leq U''(r)$ for any $r\in(0,R]$. From this and \eqref{10mareq2*} we obtain that  $u(x)=U(|x|)$ is a classical solution of 
$$
-\lambda_1(D^2u)+|Du|^p=f(|x|) \ \ \IN B_R.
$$

\medskip
\noindent
Case (c). Suppose $f=C$ with $C\geq0$. Then it turns out that $-Cr\leq s(r)\leq0$. By \eqref{9mareq10bis} we have
$$
0\geq\frac{s(r)}{r}=\big(1-rp|s(r)|^{p-2}s(r))\big)s'(r).
$$
From this inequality we deduce that  $s'(r)\leq0$ and that condition \eqref{f=constant} still holds. Thus we conclude as in the previous case. 
\end{proof}
\begin{rem}
\rm If $f$ is merely a continuous function satisfying \eqref{9mareq11}, we can choose a function $g\in C^1([0,R])$ such that 
$$
-\frac{p-1}{(pr)^{\frac{p}{p-1}}}< g(r)\leq f(r) \ \ \text{for $r\in(0,R)$.}
$$
Then Proposition \ref{4marprop2}, with $f$ replaced by $g$, yields the existence of classical subsolutions to $$-\lambda_i(D^2u)+|Du|^p=f(|x|) \,\quad\text{in $B_R$}$$ whenever $1<i\leq N$ and also for $i=1$ if $g$ can be {chosen} in such a way conditions (b) or (c) of Proposition \ref{4marprop2} are satified. As for the supersolutions, we observe that if $f\in C(B_R)$, not necessarily radial, and bounded from above, then the function $$u(x)=-\frac{\left\|f^+\right\|_\infty}{2}|x|^2$$
satisfies  {$D^2u(x)=-\left\|f^+\right\|_\infty I_N$} for any $x\in\overline B_R$. Hence, for any $1\leq i\leq N$, it is a classical supersolution of $$-\lambda_i(D^2u(x))+{|Du(x)|}^p\geq-\lambda_i(D^2u(x))\geq f(x) \,\quad\text{in $B_R$.}$$
\end{rem}

\begin{rem}
\rm The results of Proposition \ref{4marprop2} can be extended to the equation
\begin{equation}\label{29giu26seq1}
-\sum_{i=1}^{N-1}a_i\lambda_i(D^2u)+{|Du|}^p=f(|x|)\quad\;
	\text{in $B_R$}
\end{equation}

Precisely, assuming that the strictly inequality holds in \eqref{12mareq2}, then:
\begin{enumerate}
	\item[(a')] if $a_1=0$, there exists a $C^2$-radial solution of \eqref{29giu26seq1};
	\item[(b')] if $a_1=0$ there exists a $C^2$-radial solution of \eqref{29giu26seq1} if either $f(r)\leq0,\;f'(r)\leq0$ in $[0,R]$  or if $f$ is a nonnegative constant function.
	\end{enumerate}
	Details are left to the interested reader.
\end{rem}

\begin{pro}\label{propLN}
Let $f\in C^1([0,R])$ be a nonincreasing function such that
\begin{equation}\label{11mareq1}
-{\left(\frac1R\int_{0}^{+\infty}\frac{dt}{1+t^p}\right)}^\frac{p}{p-1}\leq f(r)\leq0\,,\;\quad \forall r\in[0,R].
\end{equation} 
Then there exists $u\in C^2(B_R)$, a  radial solution of 
\begin{equation}\label{11mareq2}
-\lambda_N(D^2u)+|Du|^p=f(|x|)\quad\;\text{in $B_R$}.
\end{equation}
Furthermore $u$ is radially nondecreasing and convex.
\end{pro}
\begin{proof}
If $f(r)=0$ for any $r\in[0,R]$, then any constant function is a trivial solution of \eqref{11mareq2}. Suppose now that $f$ is not constant. In view of \eqref{11mareq1} we have
\begin{equation}\label{11mareq3}
0<\left\|f\right\|_\infty\leq{\left(\frac1R\int_{0}^{+\infty}\frac{dt}{1+t^p}\right)}^\frac{p}{p-1}.
\end{equation}
Consider the initial value problem
\begin{equation*}
\left\{
\begin{array}{c}
v'(r)=|v(r)|^p-f(r) \\
v(0)=0.
\end{array}\right.
\end{equation*}
Since $p>1$, such a problem admits a unique solution $v\in C^1([0,\rho))$, in fact $v\in C^2([0,\rho))$,  defined on the maximal interval $[0,\rho)$ with $\rho\leq R$. Since $f(r)\leq0$, then $v'(r)\geq0$ for any $r\in[0,\rho)$. By the initial condition $v(0)=0$, we obtain that $v(r)\geq0$ and 
\begin{equation}\label{11mareq4}
v'(r)={\left(v(r)\right)}^p-f(r)\,, \quad\; r\in[0,\rho).
\end{equation}
We claim that $\rho=R$. If not, since $v$ is nondecreasing, we would have 
\begin{equation}\label{11mareq4bis}
\lim_{r\to\rho^-}v(r)=+\infty.
\end{equation}
To get a contradiction, we compare $v$ with the solution $w$ of the initial value problem
\begin{equation*}
\left\{\begin{array}{c}
w'(r)=|w(r)|^p+\left\|f\right\|_\infty \\
w(0)=0.
\end{array}\right.
\end{equation*}
Since $\left\|f\right\|_\infty>0$, the function $w$ is increasing. A direct integration yields $w$ implicitly defined by the equation
\begin{equation}\label{11mareq5}
\frac{1}{{\left\|f\right\|_\infty}^\frac{p-1}{p}}\int_{0}^{\frac{w(r)}{\left\|f\right\|_\infty^{1/p}}}\frac{dt}{1+t^p}=r.
\end{equation}
Consider the function
$$
F(\tau)=\int_0^\tau\frac{dt}{1+t^p}\,,\quad\tau\geq0.
$$
Note that $F$ is smooth in $[0,+\infty)$, increasing  and $F(0)=0$. Its inverse $F^{-1}$ is defined in the interval $[0,F(\infty))$, with $F(\infty)=\int_0^{+\infty}\frac{dt}{1+t^p}$.  
Using \eqref{11mareq3}, for any $r<R$ we have
$$
\left\|f\right\|_\infty^\frac{p-1}{p}r<\left\|f\right\|_\infty^\frac{p-1}{p}R\leq F(\infty).
$$
From this and using \eqref{11mareq5} we infer that 
$$
w(r)=\left\|f\right\|_\infty^\frac1pF^{-1}\left({\left\|f\right\|_\infty}^\frac{p-1}{p}r\right)\,,\;\quad r\in[0,R).
$$
By comparison,  we obtain that $v(r)\leq w(r)$ for any $r\in[0,\rho)$. Since $\rho<R$ and $w$ is increasing, then $v(r)<w(\rho)<+\infty$, in contradiction to \eqref{11mareq4bis}.

\smallskip
Thus the function $v$ is solution of \eqref{11mareq4} in $[0,R)$. We now claim that 
\begin{equation}\label{11marpomeq1}
\frac{v(r)}{r}\leq v'(r)\quad\; \forall r\in(0,R).
\end{equation} 
It is easy to see that \eqref{11marpomeq1} is equivalent to $h(r)\geq0$, where 
$$
h(r)=r({(v(r))}^p-f(r))-v(r).
$$
Using the assumption $f'(r)\leq0$ in $(0,R)$ and \eqref{11mareq4}, we get
$$
h'(r)=r\left(p{(v(r))}^{p-1}v'(r)-f'(r)\right)\geq0\,,\quad\,r\in(0,R).
$$
It follows that $h(r)\geq h(0)=0$ as claimed.

\smallskip
Define $U\in C^2([0,R))$ by
$$
U(r)=\int_0^rv(s)\,ds\quad\,\text{for $r\in[0,R)$}
$$
and $u(x)=U(|x|)$. By \eqref{11mareq4} and \eqref{11marpomeq1}, $U$ satisfies
\begin{equation*}
\left\{\begin{array}{cl}
-U''(r)+{\left(U'(r)\right)}^p=f(r) & \text{for $r\in[0,R)$}\\
U'(0)=0 &\,\\
U''(r)\geq\frac{U'(r)}{r}>0 & \text{for $r\in(0,R)$}.
\end{array}\right.
\end{equation*}
Since $U'(0)=0$, we have $u\in C^2(B_R)$. Moreover $\lambda_N(D^2u(x))=U''(|x|)$  for $|x|\neq0$ and  $D^2u(0)=\lambda_N(0)=-f(0)I_N$.
Then $u$ is a classical solution of  \eqref{11mareq2}.
\end{proof}

We conclude this section by establishing the equivalence between the PDE and ODE formulations of viscosity radial sub- and supersolutions.

\smallskip
For $(r,x,y)\in(0,+\infty)\times\mathbb R\times\mathbb R$, define
\[
\bar\lambda_i(r,x,y)=\begin{cases}
\min\{x,y/r\} & \text{if $i=1$},\\
y/r &  \text{if $1<i<N$},\\
\max\{x,y/r\}& \text{if $i=N$}. 
\end{cases}
\]
\begin{pro}\label{eqv-rad-gen}
Let $U\in \USC((0,R))$ and let $f\in C((0,R))$. Set $u(x)=U(|x|)$ for $x\in B_R\backslash\left\{0\right\}$. 
Assume that $U$ is locally bounded in $(0,R)$. 
Then, $u$ is a viscosity subsolution of
\begin{equation}\label{N-d}
-\lambda_i(D^2u)+|Du|^p=f(|x|) \ \ \text{ in $ B_R\backslash\left\{0\right\}$}
\end{equation}
if and ony if $U$ is a viscosity subsolution of
\begin{equation}\label{1-d}
-\bar\lambda_i(r,U'',U')+|U'|^p=f(r) \ \ \text{ in $(0,R)$.}
\end{equation}
\end{pro}

A statement for supersolutions, pararell to the above, is valid.

To avoid disrupting the main focus of this article, the proof of Proposition \ref{eqv-rad-gen} will be included in the appendix.

\section{The case $\lambda_i$ for $i\leq N-1$}\label{theproblem}

\subsection{Nonexistence of blow-up subsolution}  In this section we give the proof of Theorem \ref{noblowupint}  together with some remarks.

\begin{proof}[Proof of Theorem \ref{noblowupint}]
Since any subsolution of  \eqref{blowuppb} is also a subsolution of $-\lambda_i(D^2u)+|Du|^p= \max\left\{0,C\right\}$ in $\Omega$, we may assume $C\geq0$ without loss of generality.

\noindent
{\bf Step 1} Construction of a radial blow-up subsolution of \eqref{blowuppb} satisfying \eqref{binfty},  starting from a viscosity subsolution $u\in \USC(\Omega)$ .

By assumption, there exists a ball $B_R(y_0)\subset\Omega$ such that $x_0\in\partial B_R(y_0)$. Arguing as in the proof of Proposition \ref{4marprop1} and using the fact that the equation in \eqref{blowuppb} is invariant by translation and rotation, the function
$$w(x)=\max_{|y-y_0|=|x-x_0|}u(y)\,,\;\quad x\in B_R(y_0)$$
is upper semicontinuous in $B_R(y_0)$ and  a subsolution of \eqref{blowuppb} . Moreover $w$ is radial around $y_0$, that is
  $w(x)=U(r)$ with $r=|x-y_0|\in[0,R)$. Taking $x=y_0+\frac rR(x_0-y_0)$  in the definition of $w$, we also have
$$
U(r)\geq u(y_0+\frac rR(x_0-y_0)) \qquad\forall r\in[0,R).
$$
Given that  $y_0+\frac rR(x_0-y_0)\to x_0$ as $r\to R^-$, by the blow-up condition \eqref{binfty} we obtain
\begin{equation}\label{11febeq3}
\lim_{r\to R^-}U(r)=+\infty.
\end{equation} 

\medskip
\noindent
{\bf Step 2} Nonexistence of radial subsolutions of
\begin{equation}\label{pom11febeq1}
-\lambda_i(D^2u)+|Du|^p=C\quad\;\text{in $B_R(y_0)$}
\end{equation}
for  $R>0$ and $y_0\in\mathbb R^N$,  satisfying
\begin{equation}\label{pom11febeq1blow}
\lim_{|x-y_0|\to R^-}u(x)=+ \infty.
\end{equation}

Let us assume by contradiction that such a function $u(x)=U(r)$, $r=|x-y_0|$ exists.
Set \[
K:=\max\Big\{x\geq 0\mid x^p-\frac{2x}{R}\leq C\Big\} \in(0,+\infty).
\]
For $0<\varepsilon<\frac R2$,  consider the radial test function
$$
\phi_\varepsilon(x)=\Phi_\varepsilon(r)=U\left(\frac R2\right)+K\left(r-\frac R2\right)+\frac{\varepsilon}{R-\varepsilon-r}\quad\text{for $r\in\left[\frac R2,R-\varepsilon\right)$}.
$$
Note that 
$$
(U-\Phi_\varepsilon)\left(\frac  R2\right)<0\;,\quad\lim_{r\to{(R-\varepsilon)}^-}(U-\Phi_\varepsilon)(r)=-\infty.
$$
We claim that 
\begin{equation}\label{29giu26peq1}
(U-\Phi_\varepsilon)(r)<0\quad\text{for $r\in\left[\frac  R2,R-\varepsilon\right)$}.
\end{equation}
To see this, we argue by contradiction and assume that \eqref{29giu26peq1} does not hold. As a consequence, there is $r_\varepsilon\in\left(\frac R2,R-\varepsilon\right)$ such that $(U-\phi_\varepsilon)(r_\varepsilon)=\max_{r\in\left[\frac R2,R-\varepsilon\right)}(U-\phi_\varepsilon)(r)$. By the viscosity property of $u$, and using the fact that $\lambda_i(D^2\phi(x))\leq\frac{\Phi'(r)}{r}$, we obtain
\begin{equation}\label{29giu26peq2}
-\frac{\Phi'_\varepsilon(r_\varepsilon)}{r_\varepsilon}+{|\Phi_\varepsilon'(r_\varepsilon)|}^p\leq C.
\end{equation}
Since  $r_\varepsilon>\frac R2$ and $\Phi'_\varepsilon>0$, we have
$$
-\frac2R\Phi'_\varepsilon(r_\varepsilon)+{\left(\Phi'_\varepsilon(r_\varepsilon)\right)}^p\leq C.
$$
By the definition of $K$, it follows that  $\Phi'_\varepsilon(r_\varepsilon)\leq K$. On the other hand, we have
$$\Phi_\varepsilon'(r_\varepsilon)=K+\frac{\varepsilon}{{(R-\varepsilon-r_\varepsilon)}^2}>K.$$ 
This contradiction proves the claim.

\medskip

By letting $\varepsilon\to0^+$ in \eqref{29giu26peq1}, we obtain  
$$
U(r)\leq U\left(\frac R2\right)+K\left(r-\frac R2\right)\quad\text{for $r\in\left[\frac R2,R\right)$}.
$$
This inequality rules out the blow-up condition $U(r)\to+\infty$ as $r\to R^-$. Hence, the nonexistence of radial blow-up subsolutions is proved.

\end{proof}

\begin{rem}
\rm We remark that the nonexistence of  radial boundary blow-up subsolutions to \eqref{blowuppb} is also a consequence of the Lipschitz continuity property of one dimensional viscosity subsolutions to the eikonal  equation $|U'|\leq K$ in $(R/2,R)$. Indeed, testing $u(x)=U(r)$ by radial smooth functions, we see that $U(r)$ is a viscosity subsolution of $-\frac{U'(r)}{r}+|U'(r)|^p\leq C$ in $(0,R)$, which implies that for some constant $K = K(C,p,R) > 0$, $U$ is a viscosity subsolution of $|U'|\leq K$ in $(R/2,R)$. Then $U$ is Lipschitz continuous in $(R/2,R)$, with Lipschitz constant $K$ (see e.g.  \cite[Proposition 1.14]{I}), and hence $U(r) \leq U(2R/3) + K|r - 2R/3|$ for $r\in (R/2,R)$, which contradicts \eqref{binfty}.
\end{rem}

\begin{rem}
\rm We point out that if, instead of  \eqref{binfty}, we consider  the blow-down condition
\begin{equation}\label{28lug26eq1}
\lim_{x\to x_0}u(x)=-\infty,
\end{equation}
 then the equation $$
-\lambda_i(D^2u)+|D u|^p=C \quad\,\text{in $\Omega$}
$$ admits solutions satisfying \eqref{28lug26eq1}, at least for  specific domains $\Omega$ and boundary points $x_0\in\partial\Omega$, provided that $p\in(1,2]$ and $i \in \{2, \ldots, N\}$.  For example, this occurs for the punctured ball $\Omega = B_R(0) \setminus \{0\}$, with $R$ sufficiently small, taking $x_0 = 0$ (see \cite[Proposition 3.4]{BGR}). \\
If, on the contrary, $i=1$, the above equation cannot admit subsolutions satisfying \eqref{28lug26eq1} provided that the interior sphere condition holds at $x_0\in\partial\Omega$. This is a direct consequence of the fact that subsolutions of $-\lambda_1(D^2u)=C$ in $B_R(y_0)\subset\Omega$, with $x_0\in\partial B_R(y_0)$, are semiconvex in $B_R(y_0)$ (see Remark \ref{convexity}) and are thus bounded from below in $B_R(y_0)$, which prevents \eqref{28lug26eq1}.
\end{rem}

\subsection{On the Dirichlet problem} 
Let $\Omega\subset\mathbb R^N$ be a bounded domain. For $p>1$ and $\gamma>0$,  consider the Dirichlet problems
\begin{equation}\tag{DP$_0$}\label{DP0}
\left\{\begin{array}{cl}
-\lambda_i(D^2u)+|D u|^p=f(x) & \text{in $\Omega$}\\
u=0 & \text{on $\partial \Omega$}
\end{array}\right.
\end{equation}
and 
\begin{equation}\tag{DP$_\gamma$}\label{DP}
\left\{\begin{array}{cl}
-\lambda_i(D^2u)+|D u|^p+\gamma u=f(x) & \text{in $\Omega$}\\
u=0 & \text{on $\partial\Omega$.}
\end{array}\right.
\end{equation}
The aim of this section is to show that, from the viewpoint of the existence of solutions, problems \eqref{DP} and \eqref{DP0} are equivalent, at least in the model case $\Omega = B_R$ and $f = -C$, with $C$ a positive constant. This phenomenon contrasts with the uniformly elliptic framework, in which \eqref{DP}, associated with a uniformly elliptic operator $F(D^2 u)$, admits a unique solution (as in the subquadratic case $p \le 2$), while  \eqref{DP0} does not.

\begin{proof}[Proof of Proposition \ref{DPvsDP0}]
If $\underline u\in C(\overline B_R)$ is solution of \eqref{DP0}, it cannot have any local maximum in $B_R$ since $C>0$. Hence  $\underline u\leq0$ in $B_R$ and $\underline u=0$ on $\partial B_R$. Moreover, using $\gamma>0$, $\underline u$ is in turn a subsolution to \eqref{DP}. Since $\overline u=0$ is a trivial supersolution to \eqref{DP}, the existence of a unique solution follows by Perron's method, see \cite[Theorem 4.1]{CIL}. \\
Suppose now that $u_\gamma$ is the solution of \eqref{DP}. We claim that $R\leq \bar R$. If the claim is true, then the conclusion follows by \cite[Proposition 3.1 and Remark 3.3]{BGR}. In order to prove the claim, note that as  consequence of the comparison principle, 
 the function $u_\gamma$ is radial, say $u_\gamma(x)=U_\gamma(|x|)$. Hence 
$$
-\lambda_i(D^2u_\gamma)+|Du_\gamma|^p=-C-\gamma U_\gamma(|x|)\quad\text{in $B_R$}
$$
and,  by Proposition \ref{4marprop1}, we obtain that 
\begin{equation}\label{pom12febeq3}
-\frac{p-1}{{\left(pr\right)}^\frac{p}{p-1}}\leq-C-\gamma U_\gamma(r)\quad\,\text{for $r\in(0,R)$}.
\end{equation}
Using that $U_\gamma(r)\to0$ as $r\to R^-$, by \eqref{pom12febeq3} we deduce
$$
-\frac{p-1}{{\left(pR\right)}^\frac{p}{p-1}}\leq -C,
$$
which is the condition $R\leq\bar R$.

\smallskip
To complete the proof, we observe that if $R\leq\bar R$, then, in view of \cite[Proposition 3.1 and Remark 3.3]{BGR} and the comparison principle \cite[Proposition 3.8]{BGR}, problem \eqref{DP0} admits a unique solution. As a consequence,  \eqref{DP} also admits a  solution which is unique, again by the comparison principle.  Conversely, if \eqref{DP0}, or equivalently \eqref{DP}, admits a solution, then necessarily $R\leq \bar R$, as proved above.
\end{proof}

We conclude the section providing a sufficient condition for the well posedness of \eqref{DP}. Such a result depends on \cite[Theorem 3.6]{BGR}, which concerns the existence and uniqueness of \eqref{DP0}.

\begin{pro}\label{suffcondDP}
Let $\Omega\subset\mathbb R^N$ be a bounded domain, $f\in C(\Omega)\cap L^\infty(\Omega)$ and set 
\begin{equation}\label{barRi}
 \bar R=\frac{{(p-1)}^{\frac{p-1}{p}}}{p{\left\|f^-\right\|}_{L^\infty(\Omega)}^{\frac{p-1}{p}}}.
\end{equation} Suppose that $\Omega$ is a uniformly convex domain such that
\begin{equation}\label{uniformconvex}
\Omega=\bigcap_{y\in Y}B_R(y)\quad\text{for some $Y\subset\mathbb R^N$ and $R\leq\bar R$.} 
\end{equation}
Then the problem \eqref{DP} has a unique viscosity solution $u\in C(\overline\Omega)$ for any $\gamma>0$.
\end{pro} 
\begin{proof}
In view of \cite[Theorem 3.6]{BGR}, there exist $\underline u\in C(\overline\Omega)$ and $\overline u\in C(\overline\Omega)$, respectively negative and positive solutions of 
\begin{equation*}
\left\{\begin{array}{cl}
-\lambda_i(D^2\underline u)+|D \underline u|^p=-\left\|f^-\right\|_{L^\infty(\Omega)} & \text{in $\Omega$}\\
\underline u=0 & \text{on $\partial \Omega$}
\end{array}\right.
\end{equation*} and
\begin{equation*}
\left\{\begin{array}{cl}
-\lambda_i(D^2\overline u)+|D \overline u|^p=\left\|f^+\right\|_{L^\infty(\Omega)} & \text{in $\Omega$}\\
\overline u=0 & \text{on $\partial \Omega$.}
\end{array}\right.
\end{equation*}
Since $\gamma > 0$, $\underline u$ and $\overline u$ are sub- and supersolutions of \eqref{DP}, respectively. The existence of the desired solution then follows from Perron's method.  
\end{proof}

\begin{rem}
\rm The Dirichlet problem \eqref{DP} for the operator $-\lambda_N(\cdot)$ will be addressed in Subsection \ref{DirlN}. An analogue of Proposition \ref{suffcondDP} will be established in Proposition \ref{DPl}, where, in fact, the condition on $\bar R$ will be relaxed.
\end{rem}

\section{The case $\lambda_N$}\label{ergodic}\label{lambdaN}

\subsection{An explicit radial blow-up solution} We begin this section by proving that, differently from the case of $\lambda_i$ with $i\leq N-1$, the boundary blow-up problem
\begin{equation}\label{13febeqDP}
\left\{\begin{array}{cl}
-\lambda_N(D^2u)+|D u|^p=-C & \text{in $B_R$}\\
\displaystyle\lim_{x\to\partial B_R}u(x)=+\infty  & \,
\end{array}\right.
\end{equation}
has a radial solution for a suitable choice of $C$.

\begin{pro}\label{proErgRad}
Let $p\in(1,2]$ and set $C={\left(\frac1R\int_{0}^{+\infty}\frac{dt}{1+t^p}\right)}^\frac{p}{p-1}$. Then there exist $u\in C^2(B_R)$
 radial solution of \eqref{13febeqDP}.
\end{pro}
\begin{proof}
As in the proof of Proposition \ref{propLN}, we consider the function 
$$F(s)=\int_0^s\frac{dt}{1+t^p}\,,\;\quad s\geq0.$$
The function $F(s)$ is smooth, increasing and concave in $[0,+\infty)$. Moreover $F(0)=0$ and
\begin{equation}\label{13febeq1}
\lim_{s\to+\infty}F(s)=RC^\frac{p-1}{p}.
\end{equation}
Set
$$
U(r)=C^\frac1p\int_0^rF^{-1}\left(C^\frac{p-1}{p}t\right)\,dt\,,\;\quad r\in[0,R).
$$
We claim that $U$ solves
\begin{equation}\label{13febeq3}
\left\{\begin{array}{cl}
-U''(r)+{\left(U'(r)\right)}^p=-C & \text{for $r\in(0,R)$}\\
U'(0)=0 &\,\\
U''(r)\geq\frac{U'(r)}{r}>0 & \text{for $r\in(0,R)$}\\
\displaystyle\lim_{r\to R^-}U(r)=+\infty. &
\end{array}\right.
\end{equation}
The validity of all the conditions in \eqref{13febeq3} implies that $u(x)=U(|x|)$ is a solution of \eqref{13febeqDP}.

\smallskip
Since for $ r\in[0,R)$ one has $$U'(r)=C^\frac1pF^{-1}\left(C^\frac{p-1}{p}r\right)\,,$$
it is then evident that $U'(0)=0$ and $U'(r)>0$ for any $r\in(0,R)$. Moreover, for any $r\in[0,R)$, 
$$
U''(r)=C\left(1+{\left(F^{-1}\left(C^\frac{p-1}{p}r\right)\right)}^p\right)=C+{(U'(r))}^p.
$$
As far as the condition $U''(r)\geq\frac{U'(r)}{r}$ for $r\in(0,R)$, by a direct computation and using the equation satisfied by $U$, it is easy to see that it is equivalent to
the condition $h(r)\geq0$ in $(0,R)$, where 
$$
h(r)=r\left({\left(U'(r)\right)}^p+C\right)-U'(r).
$$  
Since the function $U$ is increasing and convex in $[0,R)$, we have that
$$
h'(r)=pr{\left(U'(r)\right)}^{p-1}U''(r)\geq0\quad r\in[0,R).
$$
Then $h(r)\geq h(0)=0$ for any $r\in(0,R)$ as required.

It remains to prove the limit condition in \eqref{13febeq3}. For this, using \eqref{13febeq1} and $p\leq2$ we conclude
\begin{equation*}
\begin{split}
\lim_{r\to R^-}U(r)&=\lim_{r\to R^-}C^\frac1p\int_0^rF^{-1}\left(C^\frac{p-1}{p}t\right)\,dt
={C^{\frac{2-p}{p}}}\lim_{r\to R^-}\int_0^{F^{-1}\left(C^\frac{p-1}{p}r\right)}sF'(s)\,ds\\
&={C^\frac{2-p}{p}}\int_0^{+\infty}\frac{s}{1+s^p}\,ds=+\infty.
\end{split}
\end{equation*}
\end{proof}
\subsection{Ergodic pairs} 
In this subsection we consider  a bounded $C^2$ domain $\Omega\subset\mathbb R^N$. Given $p\in(1,2]$ and $f\in C(\Omega)\cap L^\infty(\Omega)$, we denote by $u_\gamma\in C(\overline\Omega)$, $\gamma>0$, the unique viscosity solution (see Proposition \ref{DirN}) to the Dirichlet problem
\begin{equation}\label{3juneq1}
\left\{
\begin{array}{cl}
-\lambda_N(D^2u)+\left|Du\right|^p+\gamma u=f(x) & \text{in $\Omega$}\\
u=0 & \text{on $\partial\Omega$.}
\end{array}\right.
\end{equation}

\begin{thm}\label{propergodic1}
Let $f\in C(\Omega)\cap L^\infty(\Omega)$. If there exists $w\in\USC(\overline\Omega)$ bounded subsolution of
\begin{equation}\label{17apreq2}
\left\{\begin{array}{cl}
-\lambda_N(D^2w)+{|Dw|}^p=f(x) & \text{in $\Omega$}\\
w=0 & \text{on $\partial \Omega$,}
\end{array}\right.
\end{equation}
then the solution $u_\gamma$ of \eqref{3juneq1} satisfies:
\begin{enumerate}
	\item[a)] $(u_\gamma)_{\gamma}$ is bounded in $C(\overline\Omega)$;
	\item[b)] there exists a sequence $\gamma_n\to0^+$ such that $u_{\gamma_n}$ converges  uniformly in $\overline\Omega$ to a solution of \eqref{17apreq2}.
\end{enumerate}
\end{thm}
\begin{rem}\rm
Corollary \ref{cor6lug26} provides a sufficient condition for the existence of a continuous viscosity solution to \eqref{17apreq2}.
\end{rem}
\begin{proof}[Proof of Theorem \ref{propergodic1}]
The family $(u_\gamma)_\gamma$ is bounded in $C(\overline \Omega)$. For this, let $R>0$ be such that $\Omega\subseteq B_R$ and consider the function $v=\frac{\left\|f^+\right\|_\infty}{2}(R^2-|x|^2)$. Since $D^2v=\left\|f^+\right\|_\infty I_N$, then  the function $v$ is a nonnegative classical supersolution of $-\lambda_N(D^2v)+|Dv|^p+\gamma v\geq f(x)$ in $\Omega$. Hence 
\begin{equation}\label{17apreq1}
u_\gamma(x)\leq v(x)\leq\frac{\left\|f^+\right\|_\infty}{2}R^2\quad\forall x\in\overline\Omega. 
\end{equation}
Since $w$ is a bounded subsolution of \eqref{17apreq2}, then $w-\left\|w^+\right\|_\infty$ is a nonpositive subsolution of \eqref{3juneq1}. By comparison, we obtain  the lower bound  $u_\gamma(x)\geq -(\left\|w^-\right\|_\infty+\left\|w^+\right\|_\infty)$ for any  $x\in\overline\Omega$. This proves that $(u_\gamma)_\gamma$ is uniformly bounded in $\overline\Omega$
Writing 
$d(x):=\dist(x,\partial \Omega)$ and $\Omega_\varepsilon:=\{x\in\Omega : d(x)<\varepsilon\}$
for $\varepsilon>0$, we use Proposition \ref{Nbarriers} below with $M=\max\left\{\left\|f\right\|_\infty,\sup_{\gamma>0}\left\|u_\gamma\right\|_\infty\right\}$ and $\varepsilon_0>0$ sufficiently small, to obtain the uniform bounds $$
\underline u(x)\leq u_\gamma(x)\leq\overline u(x)\quad\forall x\in\Omega_{\varepsilon_0}.
$$
By the local Lipschitz estimates of Proposition \ref{prop4regularity} or \ref{prop1regularity}, we obtain that $(u_\gamma)_{0<\gamma<1}$  is also locally uniformly Lipschitz. We claim that the collection  $(u_\gamma)_{0<\gamma<1}$ is equicontinuous on $\ol\Omega$. To see this, we choose 
an increasing continuous function $\nu$, vanishing at $0$, on $[0,\,\varepsilon_0]$ 
(recalling that $\underline{u}, \overline{u}\in C(\overline{\Omega}_{\varepsilon_0})$ 
and $\underline{u}, \overline{u}=0$ on $\partial\Omega$)
 such that $$
-\nu(d(x))\leq \underline u(x)\leq \ol u(x)\leq \nu(d(x))\quad \text{ for }x\in \ol{\Omega}_{\ep_0},
$$
and also a constant $K_\ep>0$ for each $\ep\in (0,\,\ep_0)$ such that 
$$
|u_\gamma(x)-u_\gamma(y)|\leq K_\ep|x-y| \quad \text{ for } x,y\in \Omega\setminus \Omega_{\ep}.
$$
Fix any  $\eta>0$. We select $\ep, \delta\in(0,\,\ep_0)$ such that 
$$\nu(\ep) <\eta\quad\text{and}\quad K_\ep \delta<\eta.
$$  
Let $x,y\in \Omega$ be such that $|x-y|<\delta$. 
Observe that if $x, y \in \Omega \setminus \Omega_\varepsilon$, then $ |u_\gamma(x) - u_\gamma(y)| \le K_\varepsilon \delta < \eta$; 
that if $x \in \Omega \setminus \Omega_\varepsilon$ and $y \in \Omega_\varepsilon$, there is a point $z \in \partial\Omega_\varepsilon \cap \Omega$ such that $ |x - z| \le |x - y|$ 
and therefore 
$$ |u_\gamma(x) - u_\gamma(y)| \le |u_\gamma(x) - u_\gamma(z)| + |u_\gamma(z)| + |u_\gamma(y)| \le K_\varepsilon \delta + 2\nu(\varepsilon) < 3\eta$$  
(by symmetry, the inequality  $ |u_\gamma(x) - u_\gamma(y)| < 3\eta$ 
also holds for  $x \in \Omega_\varepsilon$  and  $y \in \Omega \setminus \Omega_\varepsilon$);
and that if $x, y \in \Omega_\varepsilon$, then $|u_\gamma(x) - u_\gamma(y)| \le 2\nu(\varepsilon) < 2\eta$.
Thus, we always have $|u_\gamma(x) - u_\gamma(y)| \le 3\eta$ for $x, y \in \overline\Omega$ if $|x - y| < \delta$, which ensures the equicontinuity of the collection $(u_\gamma)_{0<\gamma<1}$ on $\overline\Omega$.
  
In view of the Arzelà–Ascoli theorem, we can select a sequence of $\gamma_n\in (0,1)$ 
converging to $0$ such that $u_{\gamma_n} \to u$ uniformly in $\ol\Omega$ for some 
$u\in C(\ol\Omega)$.  
By the stability property  of viscosity solutions (see e.g. \cite[Lemma 6.1 ]{CIL}) we infer that $u$ is a solution of  \eqref{17apreq2}.
\end{proof}

\begin{thm}\label{prop2erg}
Let $f\in Lip_{loc}(\Omega)\cap L^\infty(\Omega)$ and suppose that \eqref{17apreq2} has no viscosity solutions. Then the solution $u_\gamma$ of \eqref{13apreq1} satisfies:
\begin{enumerate}
	\item[a)] $\displaystyle\lim_{\gamma\to0^+}\left\|u_\gamma^-\right\|_\infty=+\infty$;
	\item[b)] there exist a sequence $\gamma_n\to0^+$, as $n\to+\infty$, and a nonnegative constant $C$ such that $\gamma_nu_{\gamma_n}\to-C$ locally uniformly in $\Omega$ as $n\to+\infty$;
	\item[c)] the sequence $v_{\gamma_n}=u_{\gamma_n}+\left\|u_{\gamma_n}^-\right\|_\infty$ converges locally uniformly in $\Omega$, as $n\to+\infty$, to a function $v$ satisfying
\begin{equation}\label{20apreq1}
\left\{\begin{array}{cl}
-\lambda_N(D^2v)+{|Dv|}^p=f(x)+C & \text{in $\Omega$}\\
v=+\infty & \text{on $\partial \Omega$}\\
\displaystyle\min_\Omega v=0. & 
\end{array}\right.
\end{equation}
\end{enumerate}
\end{thm}
\begin{proof}
a)\; If not, we could extract a sequence $\gamma_n\to0^+$, such that $(u_{\gamma_n})_n$ is uniformly bounded from below in $\overline\Omega$.  Hence, by \eqref{17apreq1},  $(u_{\gamma_n})_n$ is uniformly bounded in $\overline\Omega$ and, as in the proof of Theorem \ref{propergodic1}, it has a subsequence converging to a solution of \eqref{17apreq2}.

\smallskip
\noindent
b)\; Since the constant function $-\frac{\left\|f^-\right\|}{\gamma}$ is a subsolution of \eqref{13apreq1}, by comparison we obtain that $u^-\leq\frac{\left\|f^-\right\|}{\gamma}$ in $\overline\Omega$. Thus, $\gamma\left\|u^-_\gamma\right\|_\Omega$ is bounded and there exists $C$ such that 
\begin{equation}\label{30apreq2p}
\lim_{n\to+\infty}\gamma_n\left\|u^-_{\gamma_n}\right\|_\infty=C
\end{equation} 
along a sequence $\gamma_n\to0^+$. Moreover, extracting a subsequence if necessary, we may further assume that $\left\|u^-_{\gamma_n}\right\|_\infty=-u_{\gamma_n}(x_n)\to+\infty$ and $x_n\to x_0\in\overline\Omega$ (in fact $x_0\in\Omega$ as it will be shown later). 

We claim that the sequence $v_{\gamma_n}=u_{\gamma_n}+\left\|u^-_{\gamma_n}\right\|_\infty$ converges, locally uniformly in $\Omega$, to a nonnegative continuous function $v$. This claim readily implies the result. Indeed, for any fixed compact set $K\subset\Omega$, since $\left\|v_{\gamma_n}\right\|_{K,\infty}$ is bounded,  we have
\begin{equation*}
\begin{split}
\left\|\gamma_nu_{\gamma_n}+C\right\|_{K,\infty}&=\left\|\gamma_nv_{\gamma_n}+(C-\gamma_{n}\left\|u^-_{\gamma_n}\right\|_\infty)\right\|_{K,\infty}\\
&\leq\gamma_n\left\|v_{\gamma_n}\right\|_{K,\infty}+\left|C-\gamma_{n}\left\|u^-_{\gamma_n}\right\|_\infty\right|\to0\;\;\text{as $n\to+\infty$.}
\end{split}
\end{equation*}
To prove the claim, we first note that 
\begin{equation}\label{20apreq3}
-\lambda_N(D^2v_{\gamma_n})+{|Dv_{\gamma_n}|}^p+\gamma_nv_{\gamma_n}=f+\gamma_n\left\|u^-_{\gamma_n}\right\|_\infty\quad\text{in $\Omega$}.
\end{equation}
Taking into account that   $\gamma_nv_{\gamma_n}$ and $\gamma_n\left\|u^-_{\gamma_n}\right\|_\infty$ are bounded sequences, by the interior Lipschitz estimate of  Proposition \ref{prop1regularity}, we infer that the sequence $v_{\gamma_n}$ is locally equi-Lipschitz. 
Since $v_{\gamma_n}(x_n)=0$, then $v_{\gamma_n}$  is also locally bounded and it converges, locally uniformly in $\Omega$, to a nonnegative function $v$,  provided $(x_n)_n\subset\left\{x\in\Omega\,:\;d(x)\geq\varepsilon_0\right\}$, for some $\varepsilon_0>0$  small enough and any sufficiently large $n$. 
For this, 
 we choose a constant $M>0$ such that $f-\gamma_n u_{\gamma_n}\geq-M+1$ 
in $\Omega$ for all $n$ and observe that 
\begin{equation}\label{strict}
-\lambda_N(D^2u_{\gamma_n})+|Du_{\gamma_n}|^p\geq -M+1 \quad\text{ in } \Omega.
\end{equation} 
According to Proposition \ref{Nbarriers}, there exist $\varepsilon>0$
and $\underline{u}\in C^2(\Omega_{\varepsilon})\cap C(\overline{\Omega_{\varepsilon}})$ 
such that 
\begin{equation}\label{non-strict}
-\lambda_N(D^2\underline{u})+|D\underline{u}|^p\leq -M \ \ \text{ in }\Omega_{\varepsilon}, 
\end{equation}
$\underline{u}\leq -M$ \ on  $\partial \Omega_{\varepsilon}\cap\Omega$,  
and $\underline{u}=0$ \ on $\partial \Omega$. 
Noting that the functions $v_{\gamma_n}$ and $\underline u+M$ satisfy
\eqref{strict} and \eqref{non-strict}, respectively, and that 
$v_{\gamma_n}\geq 0 \geq \underline u+M$ on $\partial \Omega_\varepsilon\cap\Omega$
 and  $v_{\gamma_n}=\|u_{\gamma_n}^-\|_\infty\geq M=\underline u+M$ 
 on $\partial\Omega$ for $n$ sufficiently large, we deduce by the comparison principle 
 for strict sub and super solutions (see \cite[5.C]{CIL}) 
 that $v_{\gamma_n}\geq \underline{u}+M$ on $\overline{\Omega_\varepsilon}$ 
 for $n$ sufficiently large. We may choose a positive $\varepsilon_0<\varepsilon$ 
 so that $\underline u>-M$ in $\Omega_{\varepsilon_0}$, and we conclude that 
 for $n$ sufficiently large,  $v_{\gamma_n}>0$ in $\Omega_{\varepsilon_0}$ 
 and $x_n\in\Omega\setminus \Omega_{\varepsilon_0}$.

\smallskip
\noindent
c)\;  By the local uniform convergence of $v_n$ to $v$, letting $n\to+\infty$ in \eqref{20apreq3} and using \eqref{30apreq2p}, we deduce that $v$ is a viscosity solution of the equation in \eqref{20apreq1}. Moreover, still by the local uniform  convergence, it holds that $v(x_0)=\lim_{n\to+\infty}v_n(x_n)$, and that $\min_\Omega v=0$. As far as the boundary blow-up condition, we first recall from the proof of b) that $v_{\gamma_n}\geq \underline u+M$ in $\Omega_{\varepsilon}$ for $n$ sufficiently large. This yields, in the limit $n\to+\infty$, that $v\geq \underline u+M$  in $\Omega_{\varepsilon}$ and 
$\liminf_{\Omega\ni x\to\partial \Omega}v(x)\geq M$.  In the last inequality, which is independent of $\varepsilon$,  $M$ can be chosen as any 
positive constant larger than $\sup_{(x,n)\in\Omega\times \N}\left(\gamma_n u_{\gamma_n}(x)-f(x)\right)$. Hence, $\lim_{\Omega\ni x\to\partial\Omega}v(x)=+\infty$.
\end{proof}

\begin{rem}
We wish to emphasize once again some qualitative differences between the operators $-\lambda_N(\cdot)$ and $-\lambda_i(\cdot)$ for $i < N$. In Theorem \ref{prop2erg}, the existence of a blow-up solution to \eqref{20apreq1} is established by taking the limit as $\gamma \to 0^+$ (along a subsequence) of $u_\gamma+\left\|u_\gamma^-\right\|_\infty$, $u_\gamma$ being the solutions to \eqref{13apreq1}, provided that \eqref{17apreq2} has no solutions. The proof relies on two main tools: the interior Lipschitz estimate given by Proposition \ref{prop1regularity} and suitable barrier functions (see in particular \eqref{non-strict}).

In the case of $-\lambda_i(\cdot)$ with $i < N$, when $\Omega = B_R$ with $R$ sufficiently large and $f$ is a negative constant, Proposition \ref{DPvsDP0} ensures that \eqref{DP} has no solutions. In this sense, the assumptions of Theorem \ref{prop2erg} become empty when $-\lambda_N$ is replaced by $-\lambda_i$ ($i < N$). In a non-radial setting, even though it is not known whether \eqref{DP0} and \eqref{DP} are equivalent from the viewpoint of solvability, it is worth noting that while the local Lipschitz estimates (Proposition \ref{prop1regularity}) are true for all $i=1,\cdots,N$, the construction of barrier functions like those of \eqref{non-strict} fails for $i<N$.
Alex Gu (personal communication) provided a counterexample showing that the equivalence of solvability between the two problems \eqref{DP0} and \eqref{DP} does not hold in general. The question remains open as to when such equivalence actually holds.
\end{rem}

We now consider the problem of finding a function $u$ and a constant $C$, referred to as an ergodic constant and depending on $\Omega$, $f$ and the operator $-\lambda_N(D^2\cdot)+|D\cdot|^p$, satisfying
\begin{equation}\label{29ape26eq1}
\left\{\begin{array}{cl}
-\lambda_N(D^2u)+|Du|^p=f(x)+C & \text{in $\Omega$}\\
u=+\infty & \text{on $\partial\Omega$.}
\end{array}\right.
\end{equation}
Note that Theorem \ref{prop2erg} yields the existence of a pair $(u,C)$, with $C\geq 0$, solving the above problem in the case where \eqref{17apreq2} does not admit any solution. On the other hand, Proposition \ref{proErgRad} provides the explicit expression of a constant $C<0$ in the radial case with $f\equiv 0$, as well as the existence of a corresponding $C^2$ blow-up solution.

\smallskip

Let us consider the quantity
\begin{equation}\label{mu*}
\mu_N^*(\Omega,f)=\inf\left\{\mu\in\mathbb R\,:\;\exists\varphi\in C(\overline\Omega),\;-\lambda_N(D^2\varphi)+|D\varphi|^p\leq f+\mu\;\;\text{in $\Omega$}\right\}.
\end{equation}

\begin{rem}\label{rem-mu*}
{\rm
If $f\in C(\Omega)$ is bounded, then $\mu_N^*(\Omega,f)$ is a real constant. It is clear that for any $\psi\in C^2(\overline\Omega)$ it holds
\begin{equation}\label{23giu26eq1}
\mu_N^*(\Omega,f)\leq\sup_{x\in\Omega}\left(-\lambda_N(D^2\psi)+|D\psi|^p-f(x)\right)<+\infty.
\end{equation}
Now, let $r>0$ and $x_r\in\Omega$ be such that $B_r(x_r)\subseteq\Omega$. We claim that 
\begin{equation}\label{23giu26eq2}
\mu_N^*(\Omega,f)\geq-\sup_{B_r(x_r)}f-\left(\frac{1}{r}\int_0^{+\infty} \frac{dt}{1+t^p}\right)^{\frac{p}{p-1}}.
\end{equation}
If this is not the case, then, by the definition of $\mu_N^*(\Omega,f)$, there exist $\mu<-\left(\frac{1}{r}\int_0^{+\infty} \frac{dt}{1+t^p}\right)^{\frac{p}{p-1}}-\sup_{B_r(x_r)}f$ and a subsolution $\varphi \in C(\overline\Omega)$ of $-\lambda_N(D^2\varphi)+|D\varphi|^p\leq f+\mu$ in $\Omega$. By Proposition \ref{proErgRad}, there is a blow-up solution $u\in C^2(B_r(x_r))$ of 
$$-\lambda_N(D^2u)+|Du|^p=-\left(\frac{1}{r}\int_0^{+\infty} \frac{dt}{1+t^p}\right)^{\frac{p}{p-1}}\quad \text{in \,$B_r(x_r)$}.$$
Since $u$ is a blow-up solution, the function $\varphi-u$ in $B_r(x_r)$ 
attains  a maximum at a point $\bar x\in B_r(x_r)$. The subsolution property of $\varphi$ yields 
the  contradiction: 
\[
-\left(\frac{1}{r}\int_0^{+\infty} \frac{dt}{1+t^p}\right)^{\frac{p}{p-1}}=-\lambda_N(D^2u(\bar x))+|Du(\bar x)|^p\leq f(\bar x)+\mu\leq \sup_{B_r(x_r)}f+\mu.
\]}
\end{rem}

We have the following
\begin{thm}\label{proergconst}
Let $f\in Lip_{loc}(\Omega)\cap L^\infty(\Omega)$. Then:
\begin{enumerate}
	\item[a)] problem \eqref{29ape26eq1} has a solution;
	\item[b)] the ergodic constant $C$ is unique and $C=\mu_N^*(\Omega,f)$;
	\item[c)] if $B_{r}(x_r)\subseteq\Omega\subseteq B_R(x_R)$ for some $x_r\in\Omega,x_R\in\mathbb R^N$ with  $R\geq r>0$, it holds
	\begin{equation}\label{23giu26eq3}
	-\left(\frac{1}{r}\int_0^{+\infty} \frac{dt}{1+t^p}\right)^{\frac{p}{p-1}}-\sup_{B_r(x_r)}f\leq \mu_N^*(\Omega,f)\leq -\inf_\Omega f -\left(\frac{1}{R}\int_0^{+\infty} \frac{dt}{1+t^p}\right)^{\frac{p}{p-1}}.
	\end{equation}
	\end{enumerate}
\end{thm}
\noindent
In order to prove Theorem \ref{proergconst}-b), we need to determine the exact blow-up rate of explosive solutions. This will follow from Proposition \ref{boundbeh}. To this end, let us introduce some notation. For $p\in(1,2]$, we set 
$$
\gamma=\frac{2-p}{p-1}.
$$
Consider the function $h:\mathbb R\mapsto\mathbb R$  defined by
\begin{equation*}
h_\gamma(t)=\left\{\begin{array}{cl}
-\gamma(\gamma+1)t+(\gamma |t|)^p & \text{if $\gamma>0$}\\
-t+t^2& \text{if $\gamma=0$}
\end{array}\right.
\end{equation*}
Note that
$$
\min_{t\in\mathbb R}h_\gamma(t)=h_\gamma\left(t_{\min}\right)<0
$$
where $t_{\min}=\frac1\gamma{\left(\frac{\gamma+1}{p}\right)}^{\frac{1}{p-1}}$ if $\gamma>0$ and $t_{\min}=\frac12$ if $\gamma=0$. Hence, for any $L> h_\gamma\left(t_{\min}\right)$, there exists a unique $C_L>t_{\min}$ such that $h_\gamma(C_L)=L$. 
In particular, $C_L$ is defined for $L\geq 0$.

\begin{pro}\label{boundbeh}
Let $g\in C(\Omega)$ be bounded from below such that  
\begin{equation}\label{27apr26eq1}
\lim_{d(x)\to0}g(x){d(x)}^{\gamma+2}=L\in[0,+\infty).
\end{equation}
If $u\in \USC(\Omega)$ and $v\in \LSC(\Omega)$ are respectively viscosity sub and supersolution to
\begin{equation*}
-\lambda_N(D^2u)+|Du|^p=g(x)\quad\text{in $\Omega$}
\end{equation*}
and they satisfy the blow-up conditions $$\lim_{d(x)\to0}u(x)=\lim_{d(x)\to0}v(x)=+\infty,$$ then 
\begin{equation}\label{27apr26eq2}
\begin{array}{ll}
\displaystyle\limsup_{d(x)\to0}u(x)d(x)^\gamma\leq C_L & \text{if $\gamma>0$}\\
\displaystyle\limsup_{d(x)\to0}\frac{u(x)}{-\log d(x)}\leq C_L & \text{if $\gamma=0$}
\end{array}
\end{equation}
and
\begin{equation}\label{27apr26eq2'}
\begin{array}{ll}
\displaystyle\liminf_{d(x)\to0}v(x)d(x)^\gamma\geq C_L & \text{if $\gamma>0$}\\
\displaystyle\liminf_{d(x)\to0}\frac{v(x)}{-\log d(x)}\geq C_L & \text{if $\gamma=0$}.
\end{array}
\end{equation}
\end{pro}
\begin{proof}  
We first consider \eqref{27apr26eq2}. Observing that
$$
\lim_{d(x)\to0}g^+(x){d(x)}^{\gamma+2}=\lim_{d(x)\to0}g(x){d(x)}^{\gamma+2}=L\geq0,
$$
then, given $\varepsilon>0$, there is $\delta_\varepsilon>0$ such that
\begin{equation}\label{27apr26eq3}
h_\gamma(C_L+\varepsilon)>g^+(x){d(x)}^{\gamma+2}\quad\,\forall x\in\Omega\,:\;d(x)<\delta_\varepsilon.
\end{equation}
  In the case $\gamma>0$,  we consider the function
$$
\varphi_\delta(x)=(C_L+\varepsilon)\left(\frac{1}{(d(x)-\delta)^\gamma}-\frac{1}{(\delta_\varepsilon-\delta)^\gamma}\right)+\sup_{\left\{d(x)=\delta_\varepsilon\right\}}u^+,
$$
where $\delta<\delta_\varepsilon$ and $\delta<d(x)<\delta_\varepsilon$. 
For $\delta_\varepsilon$ sufficiently small, by a straightforward computation and using \eqref{27apr26eq3} we obtain, for $\delta<d(x)<\delta_\varepsilon$, that 
\begin{equation}
\begin{split}
-\lambda_N(D^2\varphi_\delta(x))+|D\varphi_\delta(x)|^p&=\frac{h_\gamma(C_L+\varepsilon)}{(d(x)-\delta)^{\gamma+2}}\\
&>\frac{d(x)^{\gamma+2}}{(d(x)-\delta)^{\gamma+2}}g^+(x)\geq g(x).
\end{split}
\end{equation}
Since $\varphi_\delta$ is a classical strict supersolution and $u\leq \varphi_\delta$ for $d(x)=\delta_\varepsilon$ and for $d(x)=\delta$, then $u(x)\leq \varphi_\delta(x)$ for any $x\in\Omega$ such that $\delta< d(x)<\delta_\varepsilon$. Sending $\delta\to0$, we obtain 
$$
u(x)\leq(C_L+\varepsilon)\left(\frac{1}{d(x)^\gamma}-\frac{1}{\delta_\varepsilon^\gamma}\right)+\sup_{\left\{d(x)=\delta_\varepsilon\right\}}u^+
$$
whenever $d(x)<\delta_\varepsilon$. It follows that
$$
\limsup_{d(x)\to0}u(x)d(x)^\gamma\leq C_L+\varepsilon.
$$
Since $\varepsilon$ is arbitrary, then \eqref{27apr26eq2} is proved in the case $\gamma>0$. \\
In the case $\gamma=0$, the argument is similar to the previous one. However, in this case we consider the function
$$
\varphi_\delta(x)=(C_0+\varepsilon)\left(-\log(d(x)-\delta)+\log(\delta_\varepsilon-\delta)\right)+\sup_{\left\{d(x)=\delta_\varepsilon\right\}}u^+.
$$
The details are left to the reader.

\medskip
We now prove \eqref{27apr26eq2'}. Consider the function $w(x)=v-\min_{\Omega }v+\left\|g^-\right\|_\infty$. It satisfies
\begin{equation}\label{26ape29eq1p}
-\lambda_N(D^2w(x))+|Dw(x)|^p+w(x)\geq g(x)+\left\|g^-\right\|_\infty\quad\text{in $\Omega$}.
\end{equation}
Moreover, by \eqref{27apr26eq1}, we also have
$$
\lim_{d(x)\to0}\left(g(x)+\left\|g^-\right\|_\infty\right)d(x)^{\gamma+2}=L\geq0.
$$
Hence, for any $\varepsilon\in(0,C_L)$, there is $\delta_\varepsilon>0$ sufficiently small such  that
\begin{equation}\label{26ape29eq2p}
h_\gamma(C_L-\varepsilon)+C_L(2\delta_\varepsilon)^2<\left(g(x)+\left\|g^-\right\|_\infty\right)d(x)^{\gamma+2}\quad\,\forall x\in\Omega\,:\;d(x)<\delta_\varepsilon.
\end{equation}
For any $0<\delta<\delta_\varepsilon$ and $d(x)<\delta_\varepsilon$ consider
$$
\varphi_\delta(x)=(C_L-\varepsilon)\left(\frac{1}{(d(x)+\delta)^\gamma}-\frac{1}{(\delta_\varepsilon+\delta)^\gamma}\right)
$$
By \eqref{26ape29eq2p}, it turns out that 
\begin{equation*}
\begin{split}
-\lambda_N(D^2\varphi_\delta(x))+|D\varphi_\delta(x)|^p+\varphi_\delta(x)&\leq\frac{h_\gamma(C_L-\varepsilon)+C_L(2\delta_\varepsilon)^2}{(d(x)+\delta)^{\gamma+2}}\\
&<\frac{d(x)^{\gamma+2}}{(d(x)+\delta)^{\gamma+2}}\left(g(x)+\left\|g^-\right\|_\infty\right)\leq g(x)+\left\|g^-\right\|_\infty.
\end{split}
\end{equation*}
Taking into account that $w$ is nonnegative and it blows-up on $\partial\Omega$, while $\varphi_\delta$ is, for any fixed $\delta$, a bounded classical supersolution vanishing for $d(x)=\delta_\varepsilon$, then we infer that $w(x)\geq\varphi_\delta(x)$ for any $x\in\Omega$ such that $d(x)<\delta_\varepsilon$. By letting $\delta\to0$, we obtain 
$$
w(x)\geq(C_L-\varepsilon)\left(\frac{1}{d(x)^\gamma}-\frac{1}{\delta_\varepsilon^\gamma}\right)\quad\forall d(x)<\delta_\varepsilon.
$$
Hence 
$$
\liminf_{d(x)\to0}v(x)d(x)^\gamma=\liminf_{d(x)\to0}w(x)d(x)^\gamma\geq C_L-\varepsilon
$$
and \eqref{27apr26eq2'} follows, in the case $\gamma>0$, since $\varepsilon$ is arbitrarily small. If $\gamma = 0$, we proceed as above by considering the following function:
$$
\varphi_\delta(x)=(C_0-\varepsilon)\left(-\log(d(x)+\delta)+\log(\delta_\varepsilon+\delta)\right).
$$
\end{proof}

\begin{proof}[Proof of Theorem \ref{proergconst}]
 a) 
Since $\mu_N^*(\Omega,f)>-\infty$ by \eqref{23giu26eq2}, 
we can choose $\mu\in\R$ so that $\mu<\mu_N^*(\Omega,f)$. As a consequence, the Dirichlet problem \eqref{17apreq2}, with $f$ replaced by $f+\mu$, cannot have viscosity solutions. Then we can  apply Theorem \ref{prop2erg} to conclude the existence  of a solution to \eqref{29ape26eq1}.
 
\medskip
\noindent
b) Suppose by contradiction that there exist $C_1<C_2$ and $u_i\in C(\Omega)$, $i=1,2$, solutions to
\begin{equation*}
\left\{\begin{array}{cl}
-\lambda_N(D^2u_i)+|Du_i|^p=f(x)+C_i & \text{in $\Omega$}\\
u_i=+\infty & \text{on $\partial\Omega$.}
\end{array}\right.
\end{equation*}
Using Proposition \ref{boundbeh} with $g=f+C_i$ and $L=0$, we infer that 
\begin{equation}\label{11mageq1}
\lim_{d(x)\to0}\frac{u_1(x)}{u_2(x)}=1.
\end{equation}
We now pick $\theta\in(0,1)$ sufficiently close to 1 such that 
 \begin{equation}\label{11mageq2}
(1-\theta)\left\|f^-\right\|_\infty+\theta C_1-C_2\leq-\frac{C_2-C_1}{2}.
\end{equation}
Fix any $x_0\in\Omega$ and consider the function $v(x)=\theta u_1(x)+K$ with $K>u_2(x_0)-\theta u_1(x_0)$. By a straightforward computation and using \eqref{11mageq2}, we have
 \begin{equation}\label{11mageq3}
-\lambda_N(D^2v)+{|Dv|}^p\leq f(x)+C_2-\frac{C_2-C_1}{2}\quad\;\text{in $\Omega$.}
\end{equation}
In view of \eqref{11mageq1}, $v(x)-u_2(x)\to-\infty$ as $d(x)\to0$. Hence, there is $\Omega'\subset\subset\Omega$ such that $v,u_2\in C(\overline{\Omega'})$, $v\leq u_2$ on $\partial\Omega'$ and $x_0\in\Omega'$. By the comparison principle with strict inequality (see \cite[5.C]{CIL}), it follows from \eqref{11mageq3} that $v\leq u_2$  in $\Omega'$, leading to the contradiction $v(x_0)\leq u(x_0)$.

\smallskip
\noindent
In order to prove that $C=\mu_N^*(\Omega,f)$, with $\mu_N^*(\Omega,f)$ defined by \eqref{mu*}, let $(u,C)$ be a solution to \eqref{29ape26eq1}. For any $(\mu,\varphi)\in\mathbb R\times C(\overline\Omega)$ satisfying 
\begin{equation}\label{11mageq4}
-\lambda_N(D^2\varphi)+|D\varphi|^p\leq f(x)+\mu\quad\text{in $\Omega$},
\end{equation}
the difference $\varphi-u$ attains its maximum in $\Omega$. Moreover, for some $\Omega'\subset\subset\Omega$, one has $\operatorname*{arg\,max}_{x\in \Omega} (\varphi-u)\subset\Omega'$. 
We claim that $C\leq\mu$. If not, then $C>\mu$ and  $\varphi$ would be  a 
strict subsolution of $-\lambda_N(D^2\varphi)+|D\varphi|^p\leq f+C$ in $\Omega$. The comparison principle \cite[5.C]{CIL} yields that $\max_{\Omega'}(\varphi-u)\leq \max_{\partial \Omega'}(\varphi-u)$, which is absurd in view of the choice of $\Omega'$.
%
%
%
%
%
It follows $C\leq\mu_N^*(\Omega,f)$. Let us suppose now by contradiction that $C<\mu_N^*(\Omega,f)$. By the definition of $\mu_N^*(\Omega,f)$, for $\mu\in\left(C,\mu_N^*(\Omega,f)\right)$ the problem
$$\left\{
\begin{array}{cl}
-\lambda_N(D^2u)+{|Du|}^p =f(x)+\mu & \text{in $\Omega$}\\
u=0 & \text{on $\partial\Omega$}
\end{array}\right.
$$ 
cannot admit a solution. Theorem \ref{prop2erg} yields the existence of a nonnegative ergodic constant $C'$ associated to the right-hand side $f+\mu$. The uniqueness of the ergodic constant leads to the contradiction $C=\mu+C'\geq\mu$.

\medskip
\noindent
 c) The first inequality in \eqref{23giu26eq3} is just  \eqref{23giu26eq2}. As far as the upper bound of $\mu_N^*(\Omega,f)$ is concerned, by Proposition \ref{proErgRad} and Theorem \ref{proergconst}-b), we see that \ $\mu_N^*(B_R(x_R),0)=-C_R$ where, for notational convenience, we set  $C_R=\left(\fr{1}{R}\int_0^\infty \fr{dt}{1+t^p}\right)^{\fr{p}{p-1}}$. Since $\Omega\subseteq B_R(x_R)$, we have\ $\mu_N^*(\Omega,0)\leq \mu_N^*(B_R(x_0),0)=-C_R$. 
For $\mu>-C_R$, there is a subsolution $\varphi\in C(\overline{B_R(x_R)})$ of 
$\,-\lambda_N(D^2\varphi)+|D\varphi|^p\leq \mu$ in $B_R(x_R)$. This implies that $\,-\lambda_N(D^2\varphi)+|D\varphi|^p\leq f-\inf_{\Omega}f+\mu$ in $\Omega$. Thus, $-\inf_{\Omega}f+\mu\geq \mu_N^*(\Omega,f)$ and by the arbitrariness of $\mu>-C_R$, we conclude that  $\mu_N^*(\Omega,f)\leq -\inf_{\Omega}f-C_R$.

\end{proof}

\begin{rem}
We consider the quantities 
\begin{equation}\label{mui}
\mu_i^*(\Omega,f)=\inf\left\{\mu\in\mathbb R\,:\;\exists\varphi\in C(\overline\Omega),\;-\lambda_i(D^2\varphi)+|D\varphi|^p\leq f+\mu\;\;\text{in $\Omega$}\right\}
\end{equation}
for any $i=1,\ldots,N$. It follows immediately from the definition that
\begin{equation}\label{20mag26eq1}
\mu^*_N(\Omega,f)\leq\mu^*_{N-1}(\Omega,f)\leq...\leq\mu^*_1(\Omega,f).
\end{equation}
As in Remark \ref{rem-mu*}, if $f\in C(\Omega)$ is supposed to be bounded from below, then for any $\psi\in C^2(\overline\Omega)$ it holds
\begin{equation}\label{23giu26eq4'}
\mu_i^*(\Omega,f)\leq\sup_{x\in\Omega}\left(-\lambda_i(D^2\psi)+|D\psi|^p-f(x)\right)<+\infty.
\end{equation}
By appropriately choosing $\psi$ in \eqref{23giu26eq4'}, one can provide more explicit bounds for $\mu_i^*(\Omega,f)$, in terms of $f$, $p$ and $\Omega$. For instance, supposing that $\Omega\subseteq B_R(x_R)$  for some $R>0$ and $x_R\in\mathbb R^N$, and considering $\psi(x)=\frac{k}{2}|x-x_R|^2$ with $k>0$ to be determined, a straightforward computation yields for $x\in B_R(x_R)$
$$
-\lambda_i(D^2\psi(x))+{|D\psi(x)|}^p=-k+{|k(x-x_R)|}^p\leq-k+{(kR)}^p.
$$
Minimizing the right-hand side of the above inequality, i.e. by selecting $k=\frac{1}{(pR^p)^{\frac{1}{p-1}}}$, we have
$$
-\lambda_i(D^2\psi(x))+{|D\psi(x)|}^p\leq-\frac{p-1}{(pR)^\frac{p}{p-1}}\quad\text{in $\Omega$}.
$$
By \eqref{23giu26eq4'} it follows that 
\begin{equation}\label{23giu26eq6}
\mu_i^*(\Omega,f)\leq-\frac{p-1}{(pR)^\frac{p}{p-1}} -\inf_\Omega f.
\end{equation}
Then, in view of  \eqref{23giu26eq2} (assuming either $B_r(x_r)\subset\subset \Omega$ or $B_r(x_r)\subseteq\Omega$  and $f$ bounded),  it turns out that $\mu^*_i(\Omega,f)$ are real constant for every $i=1,\ldots,N$. However, it is worth noting that only $\mu_N^*$ is actually an 
ergodic constant (see Theorem \ref{proergconst}) since, by Theorem \ref{noblowupint}, for any $i=1,\ldots,N-1$ the corresponding blow-up boundary value problem
\begin{equation*}
\left\{\begin{array}{cl}
-\lambda_i(D^2u)+|Du|^p=f(x)+C & \text{in $\Omega$}\\
u=+\infty & \text{on $\partial\Omega$}
\end{array}\right.
\end{equation*}
admits no solutions for any $C\in\mathbb R$ and any continuous function $f$ bounded from above in $\Omega$.

Let us conclude this remark by noting that the upper bound for $\mu_N^*(\Omega,f)$ obtained in \eqref{23giu26eq3} is in fact sharper than that in \eqref{23giu26eq6} for $i=N$, due to the inequality \begin{equation}\label{29giu26eq1}\frac{p-1}{{p}^\frac{p}{p-1}}<\left(\int_0^{+\infty} \frac{dt}{1+t^p}\right)^{\frac{p}{p-1}}\end{equation}
which holds for all $R>0$ and $p\in(1,2]$. Indeed, using Gamma and Beta functions, one has $$\int_0^{+\infty} \frac{dt}{1+t^p}=\frac\pi p\frac{1}{\sin\left(\frac\pi p\right)},$$
from which \eqref{29giu26eq1} easily follows. 
Alternatively, since $p\in(1,2]$, we have $$
\int_0^{+\infty} \frac{dt}{1+t^p}\geq\frac12+\int_1^{+\infty}\frac{dt}{1+t^2}=\frac{2+\pi}{4}>1>\frac{{\left(p-1\right)}^\frac{p-1}{p}}{p}
$$
and \eqref{29giu26eq1} follows as well.
\end{rem}

We present two additional remarks. In Remark \ref{rem-ag}, we give a lower bound for $\mu_i^*(\Omega,f)$, with $i<N$, which is in fact better than the one that can be deduced simply combining \eqref{20mag26eq1} with \eqref{23giu26eq3}
In Remark \ref{rem-strip}, we improve the upper bound for $\mu_N^*(\Omega,f)$ 
in \eqref{23giu26eq3} by showing that it holds provided that $\Omega$ is contained in a strip of width $2R$.

\begin{rem}\label{rem-ag}
Inspired by the counterexample of Alex Gu mentioned in Remark 5.5, we obtained the following lower-bound estimate for $\mu^*_i(\Omega,f)$ with $i<N$. Assume $f\in C(\Omega)$ is bounded, and define 
\[ 
\beta(\Omega,f):= \inf\Big\{\fr{p-1}{(pr)^{\fr{p}{p-1}}}+\max_{\partial B_r(y)}f : r>0,\, \ol{B_r}(y)\subseteq \Omega
\Big\}. 
\]
The estimate is as follows.
\beq\label{ag}
\mu_i^*(\Omega,f)\geq -\beta(\Omega,f). 
\eeq
This is a consequence of Proposition 3.1. Indeed, let $\mu>\mu_i^*(\Omega,f)$, which implies that 
$\mu_i^*(\Omega,f+\mu)<0$. By the definition of $\mu_i^*(\Omega,f+\mu)$, there exists a subsolution 
$u\in C(\ol\Omega)$ of \,$-\lambda_i(D^2u)+|Du|^p=f+\mu$ in $\Omega$. Let $r>0$ and $y$ be such that
$\ol{B_r}(y)\subseteq \Omega$. We can choose $R>r$ so that $B_R(y)\subseteq\Omega$. Setting 
$\tilde f(t)=\max_{\partial B_t(y)}f$ for $t\in(0,R)$, we note that $\tilde f\in C((0,\,R))$, 
$f(x)\leq \tilde f(|x-y|)$ for $x\in B_R(y)\stm\{y\}$, and $u$ is a subsolution of
$-\lambda_i(D^2u)+|Du|^p=\tilde f(|x-y|)+\mu$ in $B_R(y)\stm\{y\}$. Hence, by Proposition \ref{4marprop1}, 
we have 
\[
\fr{p-1}{(pt)^{\fr{p}{p-1}}}+\tilde f(t)+\mu\geq 0 \quad \FOR t\in(0,R),
\]
and, in particular, 
\[
\fr{p-1}{(pr)^{\fr{p}{p-1}}}+\max_{\partial B_r(y)}f+\mu\geq 0.
\]
Taking infimum over all $r>0$ and $y\in\Omega$, with $\ol {B_r}(y)\subseteq\Omega$, we find that
$\beta(\Omega,f)+\mu\geq 0$, proving the claim. Assume as in \eqref{23giu26eq3} that 
$B_r(y)\subseteq \Omega$ for some $r>0$. By \eqref{23giu26eq3} and \eqref{20mag26eq1}, we have
\beq\label{previous-ineq}
\mu_i^*(\Omega,f)\geq -\left(\fr 1r\int_0^{+\infty}\fr{dt}{t^p+1} \right)^{\fr{p}{p-1}}-\sup_{B_r(y)} f.
\eeq
In this case, we have 
\[
\beta(\Omega,f)\leq \fr{p-1}{(pr)^{\fr{p}{p-1}}}+\sup_{B_r(y)}f
<\left(\fr 1r\int_0^{+\infty}\fr{dt}{t^p+1} \right)^{\fr{p}{p-1}}+\sup_{B_r(y)} f,
\]
where the last inequality is due to \eqref{29giu26eq1}, and
the estimate \eqref{ag} is better than \eqref{previous-ineq}. 
\end{rem}

\begin{rem} \label{rem-strip} Notably, restricting focus to one-dimensional (rather than radial) functions $\varphi$ in \eqref{mu*} yields a new formulation of the upper bound of $\mu_N^*(\Omega,f)$ in Theorem 
\ref{proergconst}. Let $R>0$ and $S$ be an open  strip, with width $2R$, i.e., 
\[
S=y+\{x\in\R^N : |e\cdot x|<R\},
\] 
where $y,e\in\R^N$ and $|e|=1$. The new formulation is the following. If $\Omega\subseteq S$ for some 
strip $S$, with width $2R$, then 
\[
\mu_N^*(\Omega,f)\leq -\inf_{\Omega} f -\left(\fr{1}{ R} \int_0^{+\infty}\fr{dt}{t^p+1}\right)^{\fr{p}{p-1}}.
\]

To see this, we proceed as in the proof of Proposition \ref{proErgRad}, sharing technical details, 
and first set $$
F(t):=\int_0^t \fr{ds}{|s|^p+1}\quad \FOR t\in\R.$$ 
Note that $F\in C^1(\R)$ and $F'(t)>0$ for $t\in\R$. 
We set 
$$
r_p:=\int_0^{+\infty} \fr{ds}{s^p+1}\quad\AND\quad C_R:=\left(\fr{r_p}{R}\right)^{\fr{p}{p-1}}
=\left(\fr{1}{R}\int_0^{+\infty}\fr{ds}{s^p+1}\right)^{\fr{p}{p-1}},
$$
and observe that $F^{-1}\in C^1((-r_p,\, r_p))$ and 
$$
(F^{-1})'(r)=|F^{-1}(r)|^p+1\quad \FOR r\in(-r_p,\,r_p).
$$
Define $v\in C^2((-r_p,\,r_p))$ by
\[
v(r)=\int_0^{r} F^{-1}(t) dt, 
\]
and note that 
\[
v''(r)=|v'(r)|^p+1 \quad\FOR r\in(-r_p,\,r_p). 
\]
Note also that $v(-r)=v(r)$ for $r\in(-r_p,\,r_p)$ and $\lim_{r\to r_p^-}v(r)=+\infty$. 

Fix any unit vector $e\in\R^N$ and any $R>0$. Consider the strip 
\[
S_{e,R}:=\{x\in\R^N : |e\cdot x|<R\},
\]
with width $2R$. We set 
$$\varphi(x):=C_R^{\fr{2-p}{p}}v\left(C_R^{\fr{p-1}{p}}e\cdot x\right) 
=C_R^{\fr{2-p}{p}}v\left(\fr{r_p}{R}\,e\cdot x\right) \quad\FOR x\in S_{e,R},
$$
 and observe that $(r_p/R)e\cdot x\in (-r_p,\,r_p)$ for $x\in S_{e,R}$ and hence, 
 $\varphi\in C^2(S_{e,R})$ and that for $x\in S_{e,R}$,
 \begin{gather*}
D\varphi(x)=C_R^{\fr 1p}v'\left(C_R^{\fr{p-1}{p}}e\cdot x\right)e, 
\qquad D^2\varphi(x)=C_Rv''\left(C_R^{\fr{p-1}{p}}e\cdot x\right)e\otimes e,
\\ 
\lambda_N (D^2\varphi(x))=C_Rv''\left(C_R^{\fr{p-1}{p}}e\cdot x\right)
=|D\varphi(x)|^p+C_R.
\end{gather*}
Similarly, for  $0<r<r_p$, if we set 
\[
C_{r,R}:=\left(\fr r R\right)^{\fr {p}{p-1}}\quad\AND\quad
\varphi_r(x):=C_{r,R}^{\fr{2-p}{p}}v\left(\fr{r}{R}\,e\cdot x\right)
=C_{r,R}^{\fr{2-p}{p}}v\left(C_{r,R}^{\fr{p-1}{p}}\,e\cdot x\right) \quad\FOR x\in \overline{S}_{e,R},
\] 
then $\varphi_r\in C^2(\overline S_{e,R})$ and 
\[
\lambda_N(D^2\varphi_r(x))=|D\varphi_r(x)|^p+C_{r,R} \quad \FOR x\in \overline S_{e,R}.
\]
Thus, we find that \, $\mu_N^*(S_{e,R},0)\leq -C_{r,R}$\, for any\, $r\in (0,r_p)$, which implies that\, $\mu_N^*(S_{e,R},0)\leq -C_R=-(r_p/R)^{\fr{p}{p-1}}$. 
Because the operator\, $-\lambda_N(D^2\cdot)+|D\cdot|^p$\, is translation- and rotation-invariant, 
for any strip\, $S$, with width $2R$, we have\, $\mu_N^*(S,0)\leq -(r_p/R)^{\fr{p}{p-1}}$.  
Now, it follows that if\, $\Omega$\, is a subset of a strip \,$S$, with width $2R$, then 
$$\mu_N^*(\Omega,f)\leq \mu_N^*(S,0)-\inf_{\Omega} f\leq -\inf_{\Omega} f-(r_p/R)^{\fr{p}{p-1}},$$ 
which is exactly what is to be shown. 
\end{rem}

\subsection{On the Dirichlet problem}\label{DirlN}
  This subsection is concerned with the existence and uniqueness of viscosity solutions to the equation  $-\lambda_N(D^2u)+|Du|^p+\gamma u=f(x)$ in $\Omega$, subject to either the Dirichlet boundary condition or the explosive condition $u=\infty$ on $\partial\Omega $.
\begin{pro}\label{DirN}
Let $\Omega$ be a bounded $C^2$ domain. Given $f\in C(\Omega)\cap L^\infty(\Omega)$, $\gamma>0$ and $p\in(1,2]$, the Dirichlet problem 
\begin{equation}\label{13apreq1}
\left\{
\begin{array}{cl}
-\lambda_N(D^2u)+\left|Du\right|^p+\gamma u=f(x) & \text{in $\Omega$}\\
u=0 & \text{on $\partial\Omega$}
\end{array}\right.
\end{equation}
admits a unique viscosity solution $u\in C(\overline \Omega)$.
\end{pro}
Since $\gamma>0$, the existence of a unique solution to \eqref{13apreq1} is classically obtained through the Perron's method. For this it is sufficient to construct continuous sub and supersolutions vanishing on  $\partial\Omega$. We recall that, since $\Omega$ is supposed to be a $C^2$ domain, then there is $\delta>0$ such that the distance function $d(x):={\rm dist}\,(x,\partial\Omega)$ is $C^2$ in $\Omega_\delta=\left\{x\in \Omega\,:\;d(x)<\delta\right\}$. Moreover $|Dd(x)|=1$ in $\Omega_\delta$.

\begin{pro}\label{Nbarriers}
Let $\Omega$ be a bounded $C^2$ domain, $p\in(1,2]$ and $M>0$. Then there exists $\varepsilon_0=\varepsilon_0(p,M)>0$ and there exist $\underline u\in C^2\left({\Omega_{\varepsilon_0}}\right)\cap C(\overline \Omega_{\varepsilon_0})$, $\overline u\in C^2\left({\Omega_{\varepsilon_0}}\right)\cap C(\overline \Omega_{\varepsilon_0})$ such that 
\begin{equation*}
\left\{
\begin{array}{cl}
-\lambda_N(D^2\underline u)+{\left|D\underline u\right|}^p\leq -M & \text{in $\Omega_{\varepsilon_0}$}\\
\underline u<0 & \text{in $\Omega_{\varepsilon_0}$}\\
\underline u\leq -M & \text{on $\partial\Omega_{\varepsilon_0}\cap\Omega$}\\
\underline u=0 & \text{on $\partial\Omega$}
\end{array}\right.
\end{equation*}
and 
\begin{equation*}
\left\{
\begin{array}{cl}
-\lambda_N(D^2\overline u)+{|D\overline u|}^p\geq M & \text{in $\Omega_{\varepsilon_0}$}\\
\overline u>0 & \text{in $\Omega_{\varepsilon_0}$}\\
\overline u\geq M & \text{on $\partial\Omega_{\varepsilon_0}\cap\Omega$}\\
\overline u=0 & \text{on $\partial\Omega$.}
\end{array}\right.
\end{equation*}
\end{pro}
\begin{proof}
Consider the function $\underline u(x)= -\frac12\log\left(1+\varepsilon_0^{-2}d(x)\right)$, where $\varepsilon_0$ is a positive constant to be selected. It is obvious that $\underline u$ is negative in $x\in \Omega_{\varepsilon_0}$ and vanishes on $\partial \Omega$. Moreover, for $x\in \Omega_{\varepsilon_0}$, a straightforward computation yields
$$
\left|D\underline u(x)\right|=\frac{\varepsilon_0^{-2}}{2(1+\varepsilon_0^{-2}d(x))}
$$
and 
$$
\lambda_N(D^2\underline u(x))\geq\left\langle D^2\underline u(x)Dd(x),Dd(x)\right\rangle=\frac{\varepsilon_0^{-4}}{2\left(1+\varepsilon_0^{-2}d(x)\right)^2}.
$$ 
Thus, for $x\in\Omega_{\varepsilon_0}$, one has
\begin{equation}\label{14apreq1}
-\lambda_N(D^2\underline u(x))+{\left|D\underline u(x)\right|}^p\leq-\frac{\varepsilon_0^{-4}}{2\left(1+\varepsilon_0^{-2}d(x)\right)^2}\left(1-2^{1-p}{\left(\varepsilon_0^2+d(x)\right)}^{2-p}\right)\,.
\end{equation}
Since $p\leq2$, for $x\in\Omega_{\varepsilon_0}$ we have
\begin{equation}\label{14apreq2}
1-2^{1-p}{\left(\varepsilon_0^2+d(x)\right)}^{2-p}\geq1-2^{1-p}{\left(\varepsilon_0^2+\varepsilon_0\right)}^{2-p}\geq\frac12
\end{equation}
for $\varepsilon_0$ small enough.
Then, by \eqref{14apreq1}-\eqref{14apreq2},  we obtain  that
\begin{equation*}
\begin{split}
-\lambda_N(D^2\underline u(x))+{\left|D\underline u(x)\right|}^p&\leq-\frac{\varepsilon_0^{-4}}{4\left(1+\varepsilon_0^{-2}d(x)\right)^2}\\
&\leq-\frac14\frac{\varepsilon_0^{-2}}{\left(1+\varepsilon_0\right)^2}\leq- M\quad\;\forall x\in\Omega_{\varepsilon_0},
\end{split}
\end{equation*}
provided $\varepsilon_0=\varepsilon_0(p,M)$ is sufficiently small. Moreover, reducing $\varepsilon_0$ if necessary, we can ensure that the following inequality holds:
$$
\underline u(x)=-\frac12\log(1+\varepsilon_0^{-1})\leq -M\quad\;\forall x\in\partial \Omega_{\varepsilon_0}\cap\Omega.
$$
The function $\underline  u$ is the desired subsolution.

\smallskip
As far as the supersolution $\overline u$ is concerned, set $\overline u(x)=Cd(x)$, $x\in\Omega_{\varepsilon_0}$, with $C>0$ to be selected. Let $k=k(\Omega)>0$ be such that
$D^2d(x)\leq kI_N$ in $\Omega_{\varepsilon_0}$. We have
\begin{equation}\label{13apreq2}
-\lambda_N(D^2\overline u(x))+{|D\overline u(x)|}^p\geq-kC+C^p\quad\;\forall x\in\Omega_{\varepsilon_0}.
\end{equation}
Moreover 
\begin{equation}\label{13apreq3}
\overline u(x)=C\varepsilon_0\quad\;\forall x\in\partial\Omega_{\varepsilon_0}\cap\Omega.
\end{equation}
Since $p>1$, we can choose $C$ sufficiently large such that $-kC+C^p\geq M$ and $C\varepsilon_0\geq M$. From \eqref{13apreq2}-\eqref{13apreq3} the conclusion follows.
\end{proof}

\begin{proof}[Sketch of the proof of Proposition \ref{DirN}]
By Proposition \ref{Nbarriers} with $M=\max\left\{\left\|f\right\|_\infty,\frac{\left\|f\right\|_\infty}{\gamma}\right\}$, we infer that  the functions
$$\underline U(x)=\left\{
\begin{array}{cl}
\max\left\{-\underline u,-M\right\} &  \text{if $x\in\overline\Omega_{\varepsilon_0}$}\\
-M & \text{if $x\in\Omega\backslash\overline\Omega_{\varepsilon_0}$}
\end{array}\right.\;,\quad
\overline U(x)=\left\{
\begin{array}{cl}
\min\left\{\overline u,M\right\} &  \text{if $x\in\overline\Omega_{\varepsilon_0}$}\\
M & \text{if $x\in\Omega\backslash\overline\Omega_{\varepsilon_0}$}
\end{array}\right.
$$
are respectively continuous viscosity sub and supersolution of \eqref{13apreq1}. Hence the conclusion follows from \cite[Theorem 4.1]{CIL}.
\end{proof}

The following propositions deal with explosive solutions.

\begin{pro}\label{DirNinfty}
Let $\Omega$ be a bounded $C^2$ domain. Given $\gamma>0$, $p\in(1,2]$ and $f\in C(\Omega)$ such that 
\begin{equation}\label{13mageq1}
\|f^-\|_{L^\infty(\Omega)}<+\infty\;\text{ and }\;\left\|f^+d^\frac{p}{p-1}\right\|_{L^\infty(\Omega)}<+\infty,
\end{equation}
then the  problem 
\begin{equation}\label{13mageq2}
\left\{
\begin{array}{cl}
-\lambda_N(D^2u)+\left|Du\right|^p+\gamma u=f(x) & \text{in $\Omega$}\\
u=+\infty & \text{on $\partial\Omega$}
\end{array}\right.
\end{equation}
admits a viscosity solution $U_\gamma\in C(\Omega)$. Moreover, for $\delta$ positive and sufficiently small, the following estimates hold in $\Omega$:
\begin{equation}\label{stime}
\begin{array}{cl}
\displaystyle\frac{c}{{(d(x))}^\frac{2-p}{p-1}}-\frac{c}{{\delta}^\frac{2-p}{p-1}}- \frac{D}{\gamma}\leq U_\gamma(x)\leq\frac{C}{{(d(x))}^\frac{2-p}{p-1}}+\frac{D}{\gamma}& \text{if $p\in(1,2)$}\\
\medskip
\displaystyle -c\log\left(d(x)\right)+c\log(\delta)- \frac{D}{\gamma}\leq U_\gamma(x)\leq -C\log\left(d(x)\right)+\frac{D}{\gamma}& \text{if $p=2$}
\end{array}
\end{equation}
where  $c,C,D$ are positive constants  depending on $p, \left\|f^-\right\|_{L^\infty(\Omega)},\left\|f^+d^\frac{p}{p-1}\right\|_{L^\infty(\Omega)}, \left\|d\right\|_{C^2(\Omega_{2\delta})}$.
\end{pro}
\begin{proof}
For $n\in\mathbb N$, let $f_n=\min\left\{f(x),n\right\}\in L^\infty(\Omega)$ and let $u_n$ be the unique viscosity solution of
 \begin{equation*}
\left\{
\begin{array}{cl}
-\lambda_N(D^2u)+\left|Du\right|^p+\gamma u=f_n(x) & \text{in $\Omega$}\\
u=n & \text{on $\partial\Omega$}.
\end{array}\right.
\end{equation*}
Note that $u_n=v_n+n$ where $v_n$ is the unique viscosity solution to \eqref{13apreq1} with right-hand side given by $f_n-\gamma n$.  
Moreover $u_n\leq u_{n+1}$, since $f_n\leq f_{n+1}$. \\
Let $\gamma=\frac{2-p}{p-1}$. We present the proof in the case $\gamma>0$, i.e. $p\in(1,2)$, and leave the case $p=2$ to the reader. Let $\varphi(x)=Cd(x)^{-\gamma}$ where  $C>0$ is such that 
\begin{equation}\label{13mageq3}
(C\gamma)^p-C\gamma(\gamma+1)=\left\|f^+d^\frac{p}{p-1}\right\|_{L^\infty(\Omega)}.
\end{equation}
For $\delta$ positive sufficiently small, a direct computation yields
\begin{equation}\label{13mageq4}
\begin{split}
-\lambda_N(D^2\varphi(x))+{|D\varphi(x)|}^p&=d(x)^{-\gamma-2}\left((C\gamma)^p-C\gamma(\gamma+1)\right)\\
&= d(x)^{-\frac{p}{p-1}}\left\|f^+d^\frac{p}{p-1}\right\|_{L^\infty(\Omega)}\geq f(x)\qquad \forall x\in\Omega_\delta.
\end{split}
\end{equation}
Consider now the function $$g(t)=\frac{C\gamma(\gamma-1)\delta^{-\gamma-4}}{4}(t-2\delta)^4+\frac{C\gamma(\gamma-2)\delta^{-\gamma-3}}{3}(t-2\delta)^3+\frac{C(\gamma^2-5\gamma+12)\delta^{-\gamma}}{12}.$$ Note that by construction
\begin{equation}\label{13mageq5}
g(\delta)=h(\delta)\;,\quad g'(\delta)=h'(\delta)\;,\quad g''(\delta)=h''(\delta)\;,\quad g'(2\delta)= g''(2\delta)=0,
\end{equation}
where $h(t)=Ct^{-\gamma}$. Moreover $g(t)>0$ for any $\delta\leq t\leq 2\delta$. Reducing $\delta$, if necessary, we may also suppose that $d\in C^2(\overline \Omega_{2\delta})$. In this way the function $g(d(x))$ is a positive classical solution of 
 \begin{equation*}
-\lambda_N(D^2g(d(x)))\geq-K \;\quad\text{ for }\delta\leq d(x)\leq 2\delta
\end{equation*}
for some $K>0$ sufficiently large. Taking $\displaystyle D=K+\left\|f^+\right\|_{L^\infty(\Omega\backslash\Omega_\delta)}$, we also have that the function $\psi(x)=g(d(x))+\frac D\gamma$ satisfies
\begin{equation}\label{13mageq6}
-\lambda_N(D^2\psi(x))+|D\psi|^p+\gamma\psi(x)\geq f(x)\;\quad\text{ for }\delta\leq d(x)\leq 2\delta.
\end{equation}
Using \eqref{13mageq4}-\eqref{13mageq5}-\eqref{13mageq6} we infer that the function
$$
\overline w=\left\{\begin{array}{cl}
\varphi(x)+\frac D\gamma & \text{for $d(x)<\delta$}\\
\psi(x) & \text{for $\delta\leq d(x)< 2\delta$}\\
\psi(2\delta)& \text{for $ d(x)\geq 2\delta$}
\end{array}\right.
$$ 
is a classical explosive solution of $-\lambda_N(D^2\overline w)+|D\overline w|^p+\gamma \overline w\geq f(x)$. By comparison we obtain that 
\begin{equation}\label{13mageq7}
u_n(x)\leq \overline w(x)\quad\;\forall x\in\Omega,\;\forall n\in\mathbb N.
\end{equation}
Now set  $c=\gamma^{-1}{\left(\frac{\gamma+1}{4}\right)}^{\frac{1}{p-1}}$ and consider, for $s>0$, the function
$$
\varphi_s(x)=\frac{c}{{(d(x)+s)}^\gamma}.
$$
For $\delta$ small to be fixed and any $0<s\leq\delta$, the function $\varphi_s$ satisfies for $d(x)<\delta$ the inequality
\begin{equation*}
-\lambda_N(D^2\varphi_s(x))+{|D\varphi_s(x)|}^p+\gamma\varphi_s(x)\leq \frac{c}{{(d(x)+s)}^{\gamma+2}}\left(-\gamma(\gamma+1)+\gamma^{p}c^{p-1}+4\gamma\delta^2\right).
\end{equation*}
By taking $\delta$ small in order that $4\gamma\delta^2\leq\frac{\gamma(\gamma+1)}{4}$ and using the definition of $c$, we obtain 
\begin{equation*}
\begin{split}
-\lambda_N(D^2\varphi_s(x))+{|D\varphi_s(x)|}^p+\gamma\varphi_s(x)&\leq -\frac{c\gamma(\gamma+1)}{2{(d(x)+s)}^{\gamma+2}}\leq-\frac{c\gamma(\gamma+1)}{2{(2\delta)}^{\gamma+2}}\\&\leq-\left\|f^-\right\|_{L^\infty(\Omega)}\leq f_n(x)\quad\text{for $d(x)<\delta$}
\end{split}
\end{equation*}
provided $\delta$ is sufficiently small.\\ On the other hand the constant function $\frac{c}{{(\delta+s)}^\gamma}-\frac{c}{{\delta}^\gamma}-\frac{\left\|f^-\right\|_{L^\infty(\Omega)}}{\gamma}$ is also a negative subsolution of  the same equation in $\Omega$. Therefore, the function 
$$
\underline w_s=\left\{\begin{array}{cl}
\varphi_s(x)-\frac{c}{{\delta}^\gamma}-\frac{\left\|f^-\right\|_{L^\infty(\Omega)}}{\gamma} & \text{for $d(x)<\delta$}\\
\frac{c}{{(\delta+s)}^\gamma}-\frac{c}{{\delta}^\gamma}-\frac{\left\|f^-\right\|_{L^\infty(\Omega)}}{\gamma} & \text{for $ d(x)\geq \delta$}
\end{array}\right.
$$ 
is a viscosity subsolution of $-\lambda_N(D^2u)+|Du|^p+\gamma u=f_n$ in $\Omega$. By comparison we infer that for any $n\geq\frac{c}{s^\gamma}-\frac{c}{{\delta}^\gamma}-\frac{\left\|f^-\right\|_{L^\infty(\Omega)}}{\gamma}$ it holds
\begin{equation*}
\underline w_s(x)\leq u_n(x)\leq \overline w(x)\quad\;\forall x\in\Omega.
\end{equation*}
Then  $(u_n)_n$ is locally bounded in $\Omega$ for all sufficiently large $n$. In view of the local Lipschitz estimates of Proposition \ref{prop4regularity}, it converges locally uniformly in $\Omega$ to a viscosity solution $U_\gamma$ of $-\lambda_N(D^2u)+|Du|^p+\gamma u=f(x)$. Moreover $U_\gamma\geq w_s$ in $\Omega$. Sending $s\to 0^+$ we conclude that 
$$
\frac{c}{{(d(x))}^\gamma}-\frac{c}{{\delta}^\gamma}- \frac{D}{\gamma}\leq U_\gamma(x)\leq\frac{C}{{(d(x))}^\gamma}+\frac{D}{\gamma}\quad\forall x\in\Omega,
$$
where $c,C,D$ are positive constants  depending on $p, \left\|f^-\right\|_{L^\infty(\Omega)},\left\|f^+d^\frac{p}{p-1}\right\|_{L^\infty(\Omega)}, \left\|d\right\|_{C^2(\Omega_{2\delta})}$. In particular $U_\gamma$ is a blow-up solution.
\end{proof}

\begin{pro}\label{DirNinftyUniq}
Let $\Omega$ be a bounded $C^2$ domain. Given $\gamma>0$, $p\in(1,2]$ and $f\in C(\Omega)$ such that 
\begin{equation}\label{13mageq1}
\|f^-\|_{L^\infty(\Omega)}<+\infty\;\text{ and }\;\lim_{d(x)\to0}f(x){d(x)}^{\frac{p}{p-1}}=L\in[0,+\infty),
\end{equation}
then there exists a unique  viscosity solution of \eqref{13mageq2}. 
\end{pro}
\begin{proof}
The existence part follows by Proposition \ref{DirNinfty}. As far as the uniqueness of explosive solution is concerned, we argue as in \cite[Theorem II.1]{LL}. First we note that the conclusions \eqref{27apr26eq2}-\eqref{27apr26eq2'} of Proposition \ref{boundbeh} continue to hold (details are left to the reader) for the equation $-\lambda_N (D^2u)+|Du|^p+\gamma u=f(x)$ in $\Omega$, where $\gamma>0$. 
   Then, using the assumption \eqref{13mageq1}, we infer that if $u_1$, $u_2$ are solutions of \eqref{13mageq2} one has $\frac{u_1}{u_2}\to1$ as $d(x)\to0^+$. Let $\theta\in (0,1)$ and set $c_\theta=-\frac{(1-\theta)}{\gamma}\left\|f^-\right\|_{L^\infty(\Omega)}$. The function $v_\theta(x)=\theta u_1(x)+c_\theta$ satisfies $v_\theta(x)-u_2(x)\to-\infty$ as $d(x)\to0^+$. Hence there is $\delta_\theta$ positive and sufficiently small such that $v_\theta\leq u_2$ in $\left\{x\in\Omega:\,d(x)\leq\delta_\theta\right\}$. Since  
	$$
	-\lambda_N(D^2v_\theta)+|Dv_\theta|^p+\gamma v_\theta\leq f(x)\qquad\text{in $\Omega$,}
	$$
	by comparison principle we have $v_\theta\leq u_2$ in $\Omega$. Sending $\theta\to1^-$, we obtain $u_1\leq u_2$ in $\Omega$. Interchanging the roles of $u_1$ and $u_2$, we conclude that  $u_1=u_2$.
\end{proof}

In the remainder of this subsection, we no longer restrict ourselves to the case $p\leq2$, allowing instead the exponent $p$ to be any real number strictly greater than one. We provide a sufficient condition ensuring the existence and uniqueness of a continuous (up to the boundary) solution to the Dirichlet problem \eqref{13apreq1}. 

It is well known that the Dirichlet problem $-F(x,D^2u)+|Du|^p+\gamma u=f$ in $\Omega$, $u=0$ on $\partial\Omega$, is well-posed in the subquadratic case $p\in[1,2]$, for $\gamma>0$, under standard assumptions on $\Omega$, $f$ and $F$, see e.g \cite{BDL2}. By contrast, when $p>2$, a loss of the boundary condition may occur, even if $F$ is uniformly elliptic, and the Dirichlet problem can be solved only in the relaxed viscosity sense, see e.g. \cite{BDL,LL}. 
In the case where $F$ is a linear elliptic operator, with nondegeneracy of ellipticity in the normal direction on $\partial\Omega$, \cite[Theorem 2.12]{CDLP} provides a sufficient condition, in the superquadratic case $p>2$, for the existence of a continuous viscosity solution satisfying the boundary condition $u=\varphi$ pointwise, under the assumptions that $\varphi\in C^{0,\frac{p-2}{p-1}}(\partial\Omega)$ that $\gamma\inf \varphi\leq\inf f$. In the particular case $\varphi=0$, this condition reduces to $f\geq0$. In the following Proposition~\ref{DPl}, for the operator $F=-\lambda_N$, we allow $f$ to take negative values in $\Omega$, although $\Omega$ is assumed to be uniformly convex. By contrast, in \cite{CDLP} no additional geometric assumptions are imposed on the domain, which is only required to be bounded and of class $C^2$.

\begin{pro}\label{DPl}
Let $\Omega\subset\mathbb R^N$ be a bounded domain, $f\in C(\Omega)\cap L^\infty(\Omega)$ and set 
\begin{equation}\label{barRN}
 \bar R=\frac{\int_0^\infty \frac{dt}{1+t^p}}{{\left\|f^-\right\|_\infty}^{\frac{p-1}{p}}}.
\end{equation}
Suppose that $\Omega$ is a uniformly convex domain such that
\begin{equation}\label{uniformconvexN}
\Omega=\bigcap_{y\in Y}B_R(y)\quad\text{for some $Y\subset\mathbb R^N$ and $R<\bar R$.} 
\end{equation}
Then, the problem \eqref{13apreq1} has a unique viscosity solution $u\in C(\overline\Omega)$ for any $\gamma>0$ and $p>1$.
\end{pro}

\begin{rem}
{\rm In view of inequality \eqref{29giu26eq1}, the number $\bar R$ defined in \eqref{barRN} is larger than the one defined in \eqref{barRi}. Consequently, the class of uniformly convex domains satisfying \eqref{uniformconvex} is contained in the class of domains fulfilling \eqref{uniformconvexN}.}
\end{rem}
\begin{proof}[Proof of Proposition \ref{DPl}.]
 Let $R<\bar R$.  First, we show that there is a supersolution 
$w\in C(\overline\Omega)$ of
\[
-\gl_N(D^2w)+|Dw|^p=f \ \ \IN \Omega,\qquad w|_{\partial\Omega}=0. 
\] 
Fix any $y\in Y$ and $z\in \partial B_R(y)$. Set $\xi=R^{-1}(y-z)$ and define
\[
\eta_z(x)=\|f^+\|_\infty^{\fr 1p} \xi\cdot (x-z).
\]
Observe that for $x\in B_R(y)$
\[
-\gl_N(D^2\eta_z(x))+|D\eta_z(x)|^p=\|f^+\|_\infty\geq f(x) 
\]
and
\[
\eta_z(x)\geq 0 \ \ \IN B_R(y),\qquad \eta_z(z)=0.
\]
The family of the functions $\eta_z$, with $z\in \partial B_R(y)$ and $y\in Y$, is uniformly 
bounded and equi-continuous on $\overline\Omega$. Setting 
$w(x)=\inf\{\eta_z(x) : z\in\partial B_R(y),\, y\in Y\}$, we find that $w\in C(\overline\Omega)$ and that 
$w$ is a nonnegative supersolution of
\[
-\gl_N(D^2w)+|Dw|^p=f \ \  \IN \Omega, \qquad w|_{\partial\Omega}=0.
\] 
The nonnegativity implies that for any $\gamma>0$, $w$
is a supersolution of 
\[
-\gl_N(D^2w)+|Dw|^p  + \gamma w=f \ \ \IN \Omega.
\]

Next, we show that there is a subsolution of 
\[
-\gl_N(D^2u)+|Du|^p=f \ \ \IN \Omega,\qquad u|_{\partial\Omega}=0.
\]
Consider the case when $f^-\not\equiv =0$.
By Proposition \ref{propLN}, for each $y\in Y$, there is a classical solution $u=u_y\in C^2(\ol B_R(y))$ of
\[
-\gl_N(D^2u)+|Du|^p=-\|f^-\|_\infty \ \ \IN B_R(y)\ \ \AND \ \ u|_{\partial B_R(y)}=0. 
\]
It follows that $u_y\leq 0$ on $\ol B_R(y)$. 
Since $-\|f^-\|_\infty\leq f$, then $u$ is a classical subsolution of 
\[
-\gl_N(D^2u)+|Du|^p=f \ \ \IN B_R(y). 
\]
According to the proof of Proposition\ref{propLN}, we may assume that
\[
u_y(x)=u_0(x-y)=U(|x-y|) \ \ \FOR x\in\ol B_R(y).
\] 
It is clear that the family of $u_y$, with $y\in Y$, is uniformly bounded and equi-continuous on
$\overline\Omega$. We define $v\in C(\overline\Omega)$ by setting
\[
v(x)=\sup_{y\in Y}u_y(x) \ \ \FOR x\in\overline\Omega.
\]
 As is well-known, $v$ is a subsolution of 
 \[
-\gl_N(D^2v)+|Dv|^p=f \ \ \IN \Omega. 
\]
When $f^-\equiv 0$, it is obvious that $v=0$ is a subsolution of 
 \[
-\gl_N(D^2v)+|Dv|^p=f \ \ \IN \Omega. 
\]
Noting that $v\leq 0$ in $\overline\Omega$, we find that for any $\gamma>0$, $v$ is a subsolution of 
\[
-\gl_N(D^2v)+|Dv|^p+\gamma v=f \ \ \IN \Omega.
\]
For each $\gamma>0$, Perron's method and the comparison principle yields 
a solution $u_\gamma\in C(\overline\Omega)$ of  
\[
-\gl_N(D^2u)+|Du|^p+\gamma u=f \ \ \IN \Omega
\]
 satisfying
\begin{equation} \label{eq16lug26}
v\leq u_\gamma\leq w \ \ \ON \overline\Omega
\end{equation} 
\end{proof}

\begin{cor}\label{cor6lug26}
Let $\Omega, \bar R, f, p$ as in Proposition \ref{DPl}. Then there exists a viscosity solution $u\in C(\overline\Omega)$ of  
\begin{equation*}
\left\{
\begin{array}{cl}
-\lambda_N(D^2u)+\left|Du\right|^p=f(x) & \text{in $\Omega$}\\
u=0 & \text{on $\partial\Omega$.}
\end{array}\right.
\end{equation*}
\end{cor}
\begin{proof}
For $\gamma>0$, let $u_\gamma$  be the unique solution of \eqref{DirN} provided by Proposition \ref{DPl}. Proposition \ref{prop1regularity}
 and the  inequality  \eqref{eq16lug26} show that the family of the functions $u_\gamma$, with $\gamma>0$, is 
 uniformly bounded and equi-continuous on $\overline\Omega$. The stability argument yields 
 a solution $u\in C(\overline\Omega)$ of 
 \[
 -\gl_N(D^2u)+|Du|^p=f \ \ \IN \Omega,\qquad u|_{\partial\Omega}=0.  \qedhere
 \]
\end{proof}

\appendix
\section{} 

Here we provide a proof of Proposition \ref{eqv-rad-gen}. 

 \bproof {\bf The ``only if'' part: }
Assume that $u$ is a viscosity subsolution of \eqref{N-d}. Let $\phi\in C^2(0,R)$. 
Assume that $U-\phi$ has a maximum at $r_0\in(0,R)$. Then, $u(x)-\phi(|x|)$ has a maximum 
at $r_0e_1$ and $\phi(|x|)$ is $C^2$ in $B_R\stm\{0\}$. Set $\psi(x)=\phi(|x|)$. Note that
\[
D\psi(x)=\phi'(|x|)\frac{x}{|x|},\qquad D^2\psi(x)=\phi'(|x|)\frac{I_N-\bar x\otimes \bar x}{|x|}
+\phi''(|x|)\bar x\otimes \bar x,
\]
and
\[
|D\psi(r_0e_1)|=|\phi'(r_0)|,\qquad 
D^2\psi(r_0e_1)=\frac{\phi'(r_0)}{r_0}\bmat 0&0&\cdots&\cdots\\ 0&1&0&\cdots   \\
\vdots&0&\ddots & \\ 
\vdots&\vdots &&1\emat+\phi''(r_0)\bmat
1&0&\cdots&\cdots  \\ 0&0&0
&\cdots \\ \vdots& 0&\ddots& \\ \vdots&\vdots&&0 \emat,
\]

By the viscosity property of $u$, we have 
\[
\lambda_i(D^2\psi(r_0e_1))+|\phi'(r_0)|^p\leq f(r_0),
\] 
which reads
\[
-\bar\lambda_i(r_0,\phi''(r_0),\phi'(r_0))+|\phi'(r_0)|^p\leq f(r_0). 
\]
Thus, $U$ is a viscosity subsolution of \eqref{1-d}. 

{\bf  The ``if'' part: }  We now assume that $U$ is a viscosity subsolution of \eqref{1-d}. 
Let $\psi\in C^2(B_R\stm\{0\})$
and assume that $u-\psi$ has a maximum at $x_0\in B_R\stm\{0\}$. 
We may assume that $\max(u-\psi)=0$ and $x_0$ is a strict maximum point of $u-\psi$.

Fix any two intervals $(\ga,\beta),  (a,b)$ such that $r_0\in (\ga,\beta)\subset\subset (a,b)\subset\subset (0, R)$, and consider the sup-convolution 
$U_\ep$ of $U$ on $[\ga,\beta]$, i.e., 
\[
U_\ep(r)=\sup_{t\in[a,b]}\left(U(t)-\frac{1}{2\ep}(r-t)^2\right). 
\]
Fix $\delta>0$ sufficiently small.  For any $\ep>0$ sufficiently small, $U_\ep$ is a viscosity subsolution 
of
\[
-\bar\lambda_{i,\delta}(r, U_\ep'',U_\ep')+|U_\ep'|^p=f_\delta(r) \ \ \IN (\ga,\beta),
\]
where 
\[
\bar\lambda_{i,\delta}(r,x,y)=\max_{t\in[r-\delta,r+\delta]}\bar\lambda_i(t,x,y)\ \ \AND \ \ 
f_\gd(r)=\max_{t\in[r-\gd,r+\gd]} f(t)\ \ \FOR r\in(\ga,\beta),
\]
and $\delta>0$ is assumed to satisfy $\ga-\delta>0$ and $\beta+\delta<R$. 
Here, it is crucial that $U$ is locally bounded in $(0,R)$. 
Note that $U_\ep$ is semiconvex, $U_\ep\geq  U$ in $[\ga,\beta]$, and 
$\lim_{\ep\to 0^+}U_\ep(r)=U(r)$ pointwise for $r\in[\ga,\beta]$. 

Set $u_\ep(x)=U_\ep(|x|)$. $U_\ep$ can be considered as a function in $\R$ and it is then a
semiconvex function in $\R$, which is locally Lipschitz continuous. In general, if $g$ is a nondecreasing and convex function in an interval of $\R$, then $g(|x|)$ is a convex function as far as $|x|$ is in the interval. From this observation, we see that $u_\ep$ is locally semiconvex in $\ol{B_\beta}\stm B_\ga$. (Consider the decomposition of $u_\ep$ as the sum of $U_\ep(|x|)+K(|x|+|x|^2)$ 
and $-K(|x|+|x|^2)$, with $K$ sufficiently large.)

Let $x_\ep$ be a maximum point of the function $u_\ep-\psi$ 
in $\ol{B_\beta}\setminus B_\ga$. By passing to a subsequence, we may assume that 
$x_\ep \to \x$ as $\ep\to 0^+$ for some $\x\in\ol{B_\beta}\stm B_\ga$. 
By a Dini's lemma type argument, we deduce that $\x=x_0$. Indeed, to check this, we suppose that $\x\not=x_0$. Since $x_0$ is a strict maximum point of $u-\psi$ with the maximum value $0$, we have $(u-\psi)(\x)<0$, and we may choose $\gamma>0$ so that $(u-\psi)(\x)<-\gamma<0$. We can choose $\ep_0>0$ so that
$(u_{\ep_0}-\psi)(\x)<-\gamma$. We may choose $\ep_1>0$ smaller than $\ep_0$ so that $(u_{\ep_0}-\psi)(x_\ep)<-\gamma$ for $0<\ep<\ep_1$, which implies that $(u_\ep-\psi)(x_\ep)<-\gamma$ for $0<\ep<\ep_1$. 
Thus, we have  a contradiction: 
\[
0=(u-\psi)(x_0)\leq(u_\ep-\psi)(x_0)\leq (u_\ep-\psi)(x_\ep)<-\gamma \ \ \FOR \ep\in(0,\ep_1).
\] 

In what follows, we may assume that $x_\ep$ is an interior point of $\ol{B_\beta}\stm B_\ga$. 
By Jensen's lemma, there is a sequence of points $(x_k,p_k) \in \ol{B_\beta}\stm B_\ga\tim \R^N$
such that 
\[\left\{\,\bald
&\lim_{k\to\infty}(x_k,p_k)=(x_\ep,0), 
\\
&\text{$u_\ep$ has a second differential at $x_k$},
\\
&\text{$u_\ep(x)-\psi(x)-\lan p_k, x\ran$ takes  a local maximum 
at $x_k$. }
\eald\right.
\]

Fix $k$ sufficiently large. By rotation, we may assume that $x_k=r_k e_1$ for some $r_k\in(\ga,\beta)$. 
By the differentiability, we can choose $(a,b,c)\in \R\tim\R^N\tim\cS(N)$ so that
\[
u_\ep(x_k+h)=a+\lan b, h\ran+\frac{1}{2}\lan ch,h\ran+o(|h|^2). 
\]
As is well-known, the choice of $(a,b,c)$ is unique. 
Let $b=(b_1,\ldots,b_N)$ and $c=(c_{ij})$ elementwise. 
We claim that 
\[\bald
&c_{ij}=0 \ \ \IF i\not= j ,
\\
& b=b_1e_1=(b_1,0,\ldots,0), 
\\
&c_{ii}= \frac{\lan b, e_1\ran}{r_k}\ \ \IF i\geq 2,
\\
&b=D\psi(x_k)+p_k, \qquad
c\leq D^2\psi(x_k), 
\eald
\]
The last claim is an easy consequence from that $x_k$ is a local maximum point of 
$x\mapsto u_\ep(x)-\psi(x)-\lan p_k,x\ran$.
A given $h=(h_1,\ldots,h_N)=h_1e_1+\cdots+h_N e_N\in\R^N$, 
let $h^j=h-2h_je_j=\sum_{i\not=j}h_ie_i-h_je_j$. For $j\geq 2$, since 
$|x_k+h^j|^2=(r_k+h_1)^2+h_2^2+\cdots+h_N^2$, 
we have 
\[
u_\ep(x_k+h)=u_\ep(x_k+h^j)=a+\lan b,h^j\ran +\frac{1}{2} \lan ch^j,h^j\ran+o(|h|^2).
\]
Note moreover that
\[\bald
\lan b,h^j\ran&=\sum_{i\not =j}b_ih_i^j+b_jh_j^j=\sum_{i\not =j}b_ih_i-b_jh_j.
\\
\lan ch^j,h^j\ran&= \sum_{i=1}^N c_{ii}(h_i^j)^2+2\sum_{k<l}c_{kl}h_k^jh_l^j
\\
&=\sum_{i=1}^Nc_{ii}h_i^2 
+2\left(\sum_{k<l, k\not= j, l\not=j}c_{kl}h_k^jh_l^j
+\sum_{k<l=j}c_{kl}h_k^jh_l^j
+\sum_{k=j<l}c_{kl}h_k^jh_l^j
\right)
\\
&=\sum_{i=1}^Nc_{ii}h_i^2 
+2\left(\sum_{k<l, k\not= j, l\not=j}c_{kl}h_kh_l
-\sum_{k<l=j}c_{kl}h_kh_l
-\sum_{k=j<l}c_{kl}h_kh_l
\right).
\eald\]
By the uniqueness of the second order expansion, we infer that for any $j\geq 2$, 
\[\bald
b_j=\lan b,e_j\ran=0,\qquad c_{ij}=0 \ \ \FOR i\not=j. 
\eald
\]
Let $j\geq 2$ and consider $x=x(\gth)=r_k \cos\gth\, e_1+r_k\sin\gth\, e_j$ for $\gth\in\R$. Note that 
$|x(\gth)|=r_k$ and $x(0)=r_ke_1$.  
Fix any $\eta>0$. Set 
\[
\phi(x)=a+\lan b,x-x_k\ran+\frac{1}{2} \lan c(x-x_k),x-x_k\ran+\eta|x-x_k|^2
\]
Then 
\[
x\mapsto u_\ep(x)-\phi(x)
\]
has a local maximum at $x=x_k$, which implies that $\gth\mapsto \phi(x(\gth))$ has a local 
minimum at $\gth=0$. Hence,
\[
\frac{d^2}{d\gth^2}\phi(x(\gth))\Big|_{\gth=0}\geq 0.
\]
Computing that
\[\bald
\frac{d}{d\gth} \phi(x(\gth))&=\phi_{x_1}(-r_k\sin\gth)+\phi_{x_j}(r_k\cos\gth),
\\
\frac{d^2}{d\gth^2}\phi(x(\gth))&=\phi_{x_1x_1}(-r_k\sin\gth)^2+\phi_{x_jx_j}(r_k\cos\gth)^2
+2\phi_{x_1x_j}(-r_k\sin\gth)(r_k\cos\gth) 
\\
&\quad
+\phi_{x_1}(-r_k\cos\gth)+\phi_{x_j}(-r_k\sin\gth),
\eald
\]
we deduce that
\[
0\leq 
\frac{d^2}{d\gth^2}\phi(x(\gth))\Big|_{\gth=0}=\phi_{x_jx_j}(x_k)r_k^2-r_k\phi_{x_1}(x_k)
=(c_{jj}+2\eta)r_k^2-r_k b_1,
\]
and hence, 
\[
\frac{b_1}{r_k}\leq c_{jj}+2\eta \ \ \FOR j\geq 2\, \AND\, \eta>0.
\]
This proves that 
\[
\frac{\lan b,e_1\ran}{r_k}\leq c_{ii} \ \ \FOR i\geq 2. 
\]
A similar argument assures that 
\[
\frac{\lan b,e_1\ran}{r_k}\geq c_{ii} \ \ \FOR i\geq 2. 
\]
Thus, 
\[
\frac{\lan b,e_1\ran}{r_k}= c_{ii} \ \ \FOR i\geq 2.
\]
At this moment, we find that 
\[
b=p_k+D\psi(x_k),\qquad b=\lan b,e_1\ran e_1,\qquad  c=\bmat c_{11}&0&\dots& \dots&0\\ 
0& \frac{\lan b, e_1\ran}{r_k} & 0&\cdots& 0\\
\vdots &0& \frac{\lan b,e_1\ran}{r_k}& \ddots&\vdots\\
\vdots &\vdots&\ddots &\ddots&0\\
 0 &0& \cdots&0& \frac{\lan b,e_1\ran}{r_k}\\ 
\emat\leq D^2\psi(x_k).
\]

Now, we are concerned with the interpretation of the expansion
\[
u_\ep(x_k+h)=a+\lan b, h\ran+\frac{1}{2}\lan ch,h\ran+o(|h|^2)
\]
in terms of $U_\ep$. Since $x_k=r_ke_1$, we see that as $t\to 0$,
\[
U_\ep(r_k+t)=u_\ep(x_k+te_1)=a+\lan b,e_1\ran t +\frac{1}{2} c_{11}t^2+o(t^2),
\]
which reads
\[
(\lan b,e_1\ran, c_{11})\in J^2 U_\ep(r_k).
\]

By the viscosity property of $U_\ep$, we have 
\[
-\bar\lambda_{i,\delta}(r_k,c_{11},\lan b,e_1\ran)+|\lan b,e_1\ran|^p\leq f_\delta(r_k).
\]
Choosing a $t_k\in[r_k-\delta,r_k+\delta]$ such that
\[
\bar\lambda_{i,\delta}(r_k,c_{11},\lan b,e_1\ran)
=\bar\lambda_i(t_k,c_{11},\lan b,e_1\ran),
\]
and noting that
\[
\Big| \frac{\lan b,e_1\ran}{r_k}-\frac{\lan b,e_1\ran}{t_k}\Big|
\leq |b|\frac{|r_k-t_k|}{r_kt_k}\leq \left(|p_k|+|D\psi(x_k)|\right)\frac{\delta}{t_kr_k},
\]
we infer that 
\[
\bar\lambda_{i,\delta}(r_k,c_{11},\lan b,e_1\ran)
=\bar\lambda_i(r_k,c_{11}, \lan b,e_1\ran)+O(\delta).
\]
Noting that
\[
\bar\lambda_i(r_k,c_{11},\lan b,e_1\ran)=\lambda_i(c)\leq \lambda_i(D^2\psi(x_k)),
\]
we conclude that
\[
-\lambda_i(D^2\psi(x_k))+|D\psi(x_k)+p_k|^p\leq f_\delta(|x_k|)+O(\delta)
\]
Sending $k\to\infty$ yields
\[
-\lambda_i(D^2\psi(x_\ep))+|D\psi(x_\ep)|^p\leq f_\delta(|x_\ep|)+O(\delta)
\]
Sending $\delta\to 0^+$ and then $\ep\to 0^+$ proves that
\[
-\lambda_i(D^2\psi(x_0))+|D\psi(x_0)|^p\leq f(|x_0|). \qedhere
\]
\eproof

\begin{rem}
\rm As it is clear from the proof, the equivalence between the PDE and ODE formulations concerning viscosity radial sub- and supersolutions will hold as soon as the operator depends only on the eigenvalues of the Hessian and the norm of the gradient, i.e. of the form
$$F(\lambda_1(D^2u), \cdots,\lambda_N(D^2u),|Du|, u, |x|).$$
\end{rem}

\section*{Acknowledgements} We wish to express our thanks to Alex Gu of the University of Chicago for promptly sharing the counterexample mentioned in Remark 5.5.

\end{document}